\documentclass{amsart}[12pt]

\usepackage{graphicx}
\usepackage{mathptmx}
\usepackage{latexsym}
\usepackage{amsmath}
\usepackage{amsfonts}
\usepackage{amssymb, amscd}
\usepackage{amsthm}

\usepackage[all]{xy}

\usepackage[square, comma, sort&compress]{natbib}

\newtheorem{theorem}{Theorem}[section]
\newtheorem{corollary}[theorem]{Corollary}
\newtheorem{lemma}[theorem]{Lemma}
\newtheorem{proposition}[theorem]{Proposition}
\newtheorem{example}[theorem]{Example}
\theoremstyle{definition}
\newtheorem{definition}[theorem]{Definition}
\theoremstyle{remark}
\newtheorem{remark}[theorem]{Remark}
\numberwithin{equation}{section}

\newcommand{\K}{\mathbb K}

\newcommand{\A}{\mathcal{A}}
\newcommand{\HH}{\mathcal{H}}

\newcommand\F{\mathbb{F} }

\newcommand\Z{\mathbb{Z} }

\newcommand\g{\mathfrak{g} }

\newcommand\ps{ \partial_\sigma }
\newcommand{\sll}{\mathfrak{sl}_2(\K)}

\newcommand\cs{\circlearrowleft }

\DeclareMathOperator{\id}{id}

\DeclareMathOperator{\Ann}{Ann}
\DeclareMathOperator{\GCD}{gcd}\DeclareMathOperator{\sign}{sign}
\DeclareMathOperator{\Der}{Der}
\DeclareMathOperator{\Lin}{\mathcal{L}}\DeclareMathOperator{\linSpan}{LinSpan}

\begin{document}

\title[Hom-Algebras and  Hom-Coalgebras]{Hom-Algebras and  Hom-Coalgebras}

\author{Abdenacer  Makhlouf \and Sergei   Silvestrov}
\address{Abdenacer  Makhlouf, Universit\'{e} de Haute Alsace,  Laboratoire de Math\'{e}matiques, \newline Informatique et Applications, \\
4, rue des Fr\`{e}res Lumi\`{e}re  F-68093 Mulhouse, France} \email{Abdenacer.Makhlouf@uha.fr}

\address{ Sergei Silvestrov,
Centre for Mathematical Sciences,  Lund University, Box
118, \newline SE-221 00 Lund, Sweden}
\email{sergei.silvestrov@math.lth.se}

\date{December 27, 2007}



\thanks{This work was supported by the Crafoord Foundation,
The Royal Swedish Academy of Sciences, The Swedish Research Council,
The Swedish Links program of SIDA Foundation and Swedish Research Council,
The Swedish Foundation of International Cooperation in Research and High Education (STINT),
The Royal Physiographic Society in Lund, University of Mulhouse and Lund University, and the European network
Liegrits}

\maketitle
\begin{abstract}
The aim of this paper is to develop the theory of Hom-coalgebras and related structures.
After reviewing some key constructions and examples of
quasi-deform\-ations
of Lie algebras involving twisted derivations and giving rise to the
class of
quasi-Lie algebras incorporating Hom-Lie algebras, we describe the
notion and some properties of Hom-algebras
and provide examples.
We introduce Hom-coalgebra
structures, leading to the notions of Hom-bialgebra
and Hom-Hopf algebras, and prove some fundamental properties and give examples.
Finally, we define the concept of Hom-Lie
admissible Hom-coalgebra
and provide their classification based on subgroups of the symmetric
group.
\end{abstract}


\section{Introduction}

In \cite{HLS,LS1,LS2}, the authors have developed a general quasi-deformation
scheme for Lie algebras of vector fields based on use of general $\sigma$-derivations, and shown that this natural quasi-deformation scheme leads to a new broad class of non-associative algebras with twisted six terms generalized Jacobi identities instead of the usual Jacobi identities of Lie algebras. The main initial motivation for this investigation was the goal of creating a unified general approach
to examples of $q$-deformations of Witt and Virasoro algebras constructed in 1990-1992 in pioneering works  \cite{AizawaSaito,ChaiElinPop,ChaiKuLukPopPresn,ChaiIsKuLuk,ChaiPopPres,CurtrZachos1,DaskaloyannisGendefVir,
Kassel1,LiuKeQin}, where in particular it was observed that in these examples some $q$-deformations of ordinary Lie algebra Jacobi identities hold.
Motivated by these and the new examples
arising as application of the general quasi-deformation construction of \cite{HLS,LS1,LS2}
on the one hand, and
the desire to be able to treat within the same framework such well-known generalizations of Lie algebras as the color and super Lie algebras on the other hand, quasi-Lie algebras and subclasses of quasi-Hom-Lie algebras
and Hom-Lie algebras were introduced in
\cite{HLS,LS1,LS2,LS3}.
In the subclass of
Hom-Lie algebras skew-symmetry is untwisted, whereas the
Jacobi identity is twisted by a single linear map and contains three
terms as in Lie algebras,
reducing to ordinary Lie algebras when the twisting linear map is the identity map.
The main feature of quasi-Lie algebras and quasi-Hom-Lie algebras is that
both the skew-symmetry and the Jacobi identity are
twisted by several deforming twisting maps and
also the Jacobi identity in quasi-Lie and
quasi-Hom-Lie algebras in general contains six
twisted triple bracket terms. The three terms twisted Jacobi identity of Hom-Lie algebras is obtained
when the six terms in Jacobi identity of the quasi-Lie or of the quasi-Hom-Lie algebras can be combined pairwise in a suitable way. That possibility depends deeply on how the twisting maps interact with each other and with the bracket multiplication.

Quasi-Lie algebras is a broad class of non-associative algebras, defined in such a way that it
encompasses Lie algebras, Lie superalgebras, color Lie
algebras as well as numerous deformed algebras arising in
connection with twisted, discrete or deformed
generalizations and modifications of derivatives and corresponding generalizations,
discrete versions of vector fields
and differential calculus.
It turns out that, among examples fitting within the framework of
quasi-Lie algebras and subclasses of quasi-Lie and Hom-Lie algebras
belong also various quantum deformations of Lie
algebras, such as deformations and quasi-deformations of the Heisenberg
Lie algebra, $\sll$, oscillator algebras and of
other finite-dimensional Lie algebras and
infinite-dimensional Lie algebras of Witt and
Virasoro type important in Physics within the
string theory, vertex operator models, quantum
scattering, lattice models and other contexts, as
well as various classes of quadratic and sub-quadratic algebras arising in connection to
non-commutative geometry, twisted derivations and
deformed difference operators and non-commutative differential calculi.
Many such examples of algebras, which can in fact be shown
to fit within the framework of quasi-Lie algebras in one or another way, can be found for instance in  \cite{AizawaSaito,ChaiElinPop,ChaiKuLukPopPresn,ChaiIsKuLuk,ChaiPopPres,CurtrZachos1,
DamKu,DaskaloyannisGendefVir,DelEtinQuantFieldStrings,DiFranMathiSen,FLM,Fuchs1,Fuchs2,HLS,HelSil-book,
Hu,Jakob,JakobLee,Kassel1,Kabook,LeBruyn,LeBruynSmith,LeBruynSmithvdBergh,LeBruynvdBergh,LiuKeQin,
FVOlaggeomassocalgbook} and in references cited therein.

In the paper \cite{MS}, we provided a different
way for constructing a subclass of quasi-Lie
algebras, the Hom-Lie algebras, by extending the
fundamental construction of Lie algebras from
associative algebras via commutator bracket
multiplication. To this end we defined the notion
of Hom-associative algebras generalizing
associative algebras to a situation where
associativity law is twisted by a linear map, and showed that the
commutator product defined using the
multiplication in a Hom-associative algebra leads
naturally to Hom-Lie algebras. We introduced also
Hom-Lie-admissible algebras and more general
$G$-Hom-associative algebras with subclasses of
Hom-Vinberg and pre-Hom-Lie algebras,
generalizing to the twisted situation
Lie-admissible algebras, $G$-associative
algebras, Vinberg and pre-Lie algebras
respectively, and shown that for these classes of
algebras the operation of taking commutator leads
to Hom-Lie algebras as well. We constructed also
all the twistings so that the brackets
$[x_1,x_2]=2 x_2, \ [x_1,x_3]=-2 x_3, \
[x_2,x_3]=x_1$ determine a three dimensional
Hom-Lie algebra. Finally, we provided for a
subclass of twistings, the list of all
three-dimensional Hom-Lie algebras. This list
contains all three-dimensional Lie algebras for
some values of structure constants.
The notion, constructions and properties of the
enveloping algebras of Hom-Lie algebras are yet
to be properly studied in full generality. An
important progress in this direction has been
made in the recent work by D. Yau
\cite{Yau:EnvLieAlg}.
Also, in this connection it may be appropriate to mention that
in \cite{Hu}, for so called their $q$-Lie algebras, which are a special subclass
of Hom-Lie algebras, a universal enveloping algebra has been defined,
Poincare-Birhkoff-Witt basis was constructed and corresponding Poincare-Birhkoff-Witt type theorem
has been proved in detail using the reduction system technique of Bergman's Diamond Lemma.
It is important challenging problem to develop further the proper notion and theory of
universal enveloping algebras for general Hom-algebras and for general
quasi-Hom-Lie and Quasi-Lie algebras.
The recent works by Hellstr{\"o}m \cite{HellstromGrInv,HellstromGDL}, pertaining to generalizations of Diamond Lemma and of reduction system technics to more general algebraic structures of operadic type, might be useful in this respect.
The fundamentals of the formal
deformation theory and associated cohomology
structures for Hom-Lie algebras have been
considered recently by the authors in
\cite{HomDeform}. Simultaneously, D. Yau has
developed elements of homology for Hom-Lie
algebras in \cite{Yau:HomolHomLie}. These
directions of future investigation promise to be
very fruitful.

Further development of the area requires a
broader insight in various new Hom-algebraic
structures, generalizing the corresponding key structures
from the context of associative and Lie algebras.
In the present paper we review some constructions
and examples of quasi-Lie and Hom-Lie algebras and
develop the coalgebra counterpart of the notions
and results of \cite{MS}, extending in particular
in the framework of Hom-associative, Hom-Lie
algebras and Hom-coalgebras, the notions and
results on associative and Lie admissible
coalgebras obtained in \cite{GR}. In this
context, we also define structures of
Hom-bialgebras, generalized Hom-bialgebras and Hom-Hopf algebras and describe
some of their properties extending properties of
bialgebras and Hopf algebras. More specifically,
in Section \ref{sec:qhlsigma} we recall the
definitions of quasi-Lie algebras and their
subclass of quasi-Hom-Lie algebras defined in
\cite{HLS,LS1,LS2,LS3}, and then review the
method of constructing the quasi-Hom-Lie algebras via
discrete modifications of vector fields using twisted
derivations, and describe two classes of examples of multi-parameter families of algebras arising as application of this method, the quasi-Lie quasi-deformations of $\sll$ on the algebra of
polynomials in nilpotent indeterminate and
quasi-Lie algebras generalizing Witt (centerless
Virasoro) algebras via discretizations by
$\sigma$-derivations with general endomorphism
$\sigma$ of the algebra of Laurent polynomials.
In Section \ref{sec:HomalgHomcoalg}, we summarize
the relevant definitions of Hom-associative
algebra, Hom-Lie algebra, Hom-Leibniz algebra and Hom-Poisson algebra. In
this section we define the notions and describe
some of basic properties of Hom-coalgebras,
Hom-bialgebras and Hom-Hopf algebras which
generalize the classical coalgebra, bialgebra and
Hopf algebra structures. We also define the module
and comodule structure over Hom-associative
algebra or Hom-coassociative coalgebra.  We also consider a generalization of bialgebra structure in the spirit of Loday \cite{Loday}, where the algebra in no longer unital and the coalgebra is no longer counital, and extend it to our context.  In \cite{Yau:HomolHomLie}, it is shown that starting from an algebra and algebra endomorphism, one can construct a Hom-associative algebra. We extend this result  to Hom-associative coalgebras and generalized Hom-bialgebras and use it to provide some examples. The structures of Hom-coalgebra and the two kinds of Hom-bialgebra structures were already introduced in \cite{Makhlouf-Hopf}. Recently, some developments on Hom-bialgebras were made in \cite{Yau:HomBial}.
In Section 4, we introduce the concept of Hom-Lie
admissible Hom-coalgebra, describe some useful
relations between comultiplication, opposite comultiplication,
the cocommutator defined as their difference, and
their $\beta$-twisted coassociators and
$\beta$-twisted co-Jacobi sums. We also introduce
the notion of $G$-Hom-coalgebra for any subgroup
$G$ of permutation group $S_3$. We show that
$G$-Hom-coalgebras are Hom-Lie admissible
Hom-coalgebras, and also establish duality based
correspondence between classes of
$G$-Hom-coalgebras and $G$-Hom-algebras.

\section{Quasi-Hom-Lie algebras associated
with $\sigma$-derivations}\label{sec:qhlsigma}
Throughout this paper $\K$ denotes a field of characteristic zero.

In this section we first recall the definitions
of quasi-Lie algebras and their subclass of
quasi-Hom-Lie algebras defined first in
\cite{HLS,LS1,LS2}, and then review the
construction of quasi-Hom-Lie algebras of
discretizations of vector fields by twisted
derivations.

Let  $\Lin_{\K}(V)$ be the set of linear maps of
the linear space $V$ over the field $\K$.

\begin{definition} (Larsson, Silvestrov \cite{LS2}) \label{def:quasiLiealg}
A \emph{quasi-Lie algebra} is a tuple
\mbox{$(V,[
\cdot,\cdot ]_V,\alpha,\beta,\omega,\theta)$}
where
\begin{itemize}
    \item $V$ is a linear space over $\K$;
    \item $[\cdot,\cdot]_V: V\times V\to V$ is a bilinear
      map called a product or bracket in $V$;
    \item $\alpha,\beta:V\to V$, are linear maps;
    \item $\omega:D_\omega\to \Lin_{\K}(V)$ and $\theta:D_\theta\to \Lin_{\K}(V)$
      are maps with domains of definition $D_\omega,
      D_\theta\subseteq V\times V$,
\end{itemize}
such that the following conditions hold:
\begin{itemize}
      \item ($\omega$-symmetry) The product satisfies a generalized skew-symmetry condition
        $$[ x,y]_V=\omega(x,y)[ y,x]_V,
        \quad\text{ for all } (x,y)\in D_\omega ;$$
\item (quasi-Jacobi identity) The bracket satisfies a generalized Jacobi identity
    $$\circlearrowleft_{x,y,z}\big\{\,\theta(z,x)\big([\alpha(x),[ y,z]_V]_V+
    \beta[ x,[ y,z]_V]_V\big)\big\}=0,$$
     for all $(z,x),(x,y),(y,z)\in D_\theta$ and  where
$\circlearrowleft_{x,y,z}$ denotes summation over
the cyclic permutation on $x,y,z$.
\end{itemize}
\end{definition}

Note that the twisting maps in the definition of
quasi-Lie algebras are not arbitrary. For example
the axioms of quasi-Lie algebra above imply some
properties like
$$(\omega(x,y)\omega(y,x)-\id)[
x,y]=0, \text{ if } (x,y), (y,x) \in D_\omega ,$$
which follows from the computation $$[
x,y]=\omega(x,y)[ y,x]
=\omega(x,y)\omega(y,x)[ x,y].$$

The class of algebras introduced in Definition
\ref{def:quasiLiealg} incorporates as special
cases \emph{Hom-Lie algebras} and
\emph{quasi-Hom-Lie algebras} which appear
naturally in the algebraic study of
$\sigma$-derivations and related deformations of
infinite-dimensional and finite-dimensional Lie
algebras. To get the class of quasi-Hom-Lie
algebras one specifies $\theta=\omega$ and
restricts to $\alpha$ and $\beta$ satisfying the
twisting condition $[
\alpha(x),\alpha(y)]=\beta\circ\alpha[
x,y]$.

\begin{definition} \label{def:qhlalg}
A \emph{quasi-Hom-Lie algebra} is a tuple
\mbox{$(V,[
\cdot,\cdot]_V,\alpha,\beta,\omega)$} where
\begin{itemize}
    \item $V$ is a $\K$-linear space,
    \item $[\cdot,\cdot]_V:V\times V\to V$ is a bilinear
      map called a \emph{product} or
    a \emph{bracket in $V$}
    \item $\alpha,\beta:V\to V$, are linear maps,
    \item $\omega:D_\omega\to \Lin_\K(V)$ is a map
    with domain of definition $D_\omega\subseteq
    V\times V$,
\end{itemize}
such that the following conditions hold:
\begin{itemize}
    \item ($\beta$-twisting.) The map $\alpha$ is a $\beta$-twisted
      algebra homomorphism,
      i.e. $$[\alpha(x),\alpha(y)]_V
      =\beta\circ\alpha[ x,y]_V,
\quad\text{ for all } x,y\in V;$$
    \item ($\omega$-symmetry.) The product satisfies a generalized skew-symmetry condition
        $$[ x,y]_V=\omega(x,y)
        [ y,x]_V,
        \quad\text{ for all } (x,y)\in D_\omega ;$$
    \item (quasi-Hom-Lie Jacobi identity.) The bracket satisfies a generalized Jacobi identity
    $$\circlearrowleft_{x,y,z}\Big\{\,\omega(z,x)
    \Big([\alpha(x),[ y,z]_V]_V+
    \beta[ x,[ y,z]_V]_V\Big)\Big\}=0,$$
     for all $(z,x),(x,y),(y,z)\in D_\omega$.
\end{itemize}
\end{definition}

Specifying in the definition of quasi-Lie
algebras $D_\omega =V \times V$, $\beta=\id_V$ and
$\theta(x,y)=\omega(x,y)=-\id_V$ for all $(x,y)\in D_\omega=D_\theta$, we get a subclass of Hom-Lie
algebras with twisting linear map $\alpha$, which is moreover an algebra homomorphism due to $\beta$-twisting axiom. This class of Hom-Lie algebras includes Lie algebras when $\alpha = \id$. We will come back to general Hom-Lie algebras in the forthcoming sections.

Important classes and examples of Quasi-Hom-Lie
algebras and Hom-Lie algebras are obtained using
a quasi-deformation procedure of discretizing
vector fields by twisted derivations.  In this
quasi-Lie deformation procedure we start with the
Lie algebra $\g$ we wish to deform, and let
$\rho: \g\to\Der(\A)\subseteq\mathfrak{gl}(\A)$
be a representation of $\g$ in terms of
derivations on some commutative, associative
algebra with unity. The Lie structure on
$\Der(\A)$ is given by the usual commutator
bracket between linear operators. The
quasi-deformation procedure changes first the
involved derivations to $\sigma$-derivations,
i.e., linear maps $\ps:\A\to\A$ satisfying a
generalized Leibniz rule:
$\ps(ab)=\ps(a)b+\sigma(a)\ps(b)$, for all
$a,b\in\A$, and for an algebra endomorphism
$\sigma$ on $\A$. The usual commutator of the
obtained $\sigma$-derivations might be in general
not a $\sigma$-derivation. Thus in the course of
this deformation we also deform the commutator to
a $\sigma$-twisted ($\sigma$-twisted) bracket
$[\cdot,\cdot]$. The new product
$[\cdot,\cdot]$ is defined and closed
on the left $\A$-submodules $\A\cdot\ps$ of
$\Der_\sigma(\A)$, for each choice of
$\ps\in\Der_\sigma(\A)$. This is the content of
\text{Theorem \ref{thm:GenWitt}} which
establishes also a canonical Jacobi-like relation
on $\A\cdot\ps$ for $[\cdot,\cdot]$,
reducing to the ordinary Jacobi identity when
$\sigma=\id$, i.e., in the "limit" case of this
deformation scheme corresponding to the Lie
algebra $\g$. We remark that in some cases, for
instance when $\A$ is a unique factorization
domain, $\A\cdot\ps=\Der_\sigma(\A)$ for suitable
$\ps\in\Der_\sigma(\A)$ (see \cite{HLS}). In this
scheme, we have two "deformation parameters"
namely, $\A$ and $\sigma$. Note, however, that
they are not independent. Indeed, $\sigma$
certainly depends on $\A$.
The algebra structure on $\A\cdot\ps$ is then
pulled back to an abstract algebra which is then
to be viewed as the quasi-deformed version of
$\g$. It might be so that one do not actually
retrieve the original $\g$ by performing
appropriate (depending on the case considered)
limit procedure. This is because for some
"values" of the involved parameters the
representation or specific operators might
collapse, and sometimes taking the limit might
even become meaningless in these circumstances.
This is why we choose to call our deformations
quasi-deformations. Another interesting phenomena
that arises is that the pull-back might "forget
relations". That is to say that the operators in
$\A\cdot\ps$ may satisfy relations, for instance
coming from the twisted Leibniz rules, that the
abstract algebra does not satisfy.

Let  $\A$ be a commutative, associative $\K$-algebra with unity
$1$. Furthermore $\sigma$ will denote an
endomorphism on $\A$. Then by a \emph{twisted
derivation} or \emph{$\sigma$-derivation} on $\A$
we mean an $\K$-linear map $\ps:\A\to\A$ such
that a $\sigma$-twisted Leibniz rule holds:
\begin{align}\label{eq:sigmaLeibniz}
\ps(ab)=\ps(a)b+\sigma(a)\ps(b).
\end{align}
The best known $\sigma$-derivations are
$(\partial\, a)(t)=a'(t)$, the ordinary
differential operator with the ordinary Leibniz
rule, i.e., $\sigma=\id$; and $(\ps\, a)(t)=(D_q
a)(t)$, the Jackson $q$-derivation operator with
$\sigma$-Leibniz rule $(D_q\, (ab))(t)=(D_q
a)(t)b(t)+a(qt)(D_q b)(t),$ where
$\sigma=\mathbf{t}_q$ and $\mathbf{t}_q
f(t):=f(qt).$ (See \cite{HelSil-book} and
references there.)

We let $\mathfrak{D}_\sigma(\A)$ denote the set of
$\sigma$-derivations on $\A$. Fixing a
homomorphism $\sigma:\A\to\A$, an element
$\ps\in\mathfrak{D}_\sigma(\A)$, and an element
$\delta\in\A$, we assume that these objects
satisfy the following two conditions:
\begin{align}\label{eq:GenWittCond1}
    &\sigma(\Ann(\ps))\subseteq \Ann(\ps),\\
    &\ps(\sigma(a)) = \delta\sigma(\ps(a)),\quad\text{for }a\in\A,\label{eq:GenWittCond2}
\end{align} where $\Ann(\ps):=\{a\in\A\,|\, a\cdot
\ps=0\}$. Let
$\A\cdot\ps=\{a\cdot\ps\;|\;a\in\A\}$ denote the
cyclic $\A$-submodule of $\mathfrak{D}_\sigma(\A)$
generated by $\ps$ and extend $\sigma$ to
$\A\cdot\ps$ by
$\sigma(a\cdot\ps)=\sigma(a)\cdot\ps$. The
following theorem, from \cite{HLS}, introducing
an $\K$-algebra structure on $\A\cdot\ps$ making
it a quasi-Hom-Lie algebra, provides a method for
construction of various classes and examples of
quasi-Lie algebras using twisted derivations.

\begin{theorem}[Hartwig, Larsson, Silvestrov \cite{HLS}]
\label{thm:GenWitt}
If $\sigma(\Ann(\ps))\subseteq \Ann(\ps)$, then
the map $[\cdot,\cdot]_\sigma$
defined by
\begin{align} \label{eq:GenWittProdDef}
    [ a\cdot\ps,b\cdot\ps]_\sigma=(\sigma(a)\cdot\ps)\circ(b\cdot\ps)-(\sigma(b)\cdot\ps)
    \circ(a\cdot\ps),
\end{align}for $a,b\in\A$ and where $\circ$ denotes composition of maps, is a well-defined
$\K$-algebra product on the $\K$-linear space
$\A\cdot\ps$. It satisfies the following
identities for $a,b,c\in\A${\rm :}
\begin{align}\label{eq:GenWittProdFormula}
    &[ a\cdot\ps,
    b\cdot\ps]_\sigma=(\sigma(a)\ps(b)-\sigma(b)\ps(a))\cdot\ps,\\
    &[ a\cdot\ps, b\cdot\ps]_\sigma=-[ b\cdot\ps,
    a\cdot\ps]_\sigma,\label{eq:GenWittSkew}
\end{align}
and if, in addition, {\rm
(\ref{eq:GenWittCond2})} holds, we have the
deformed six-term Jacobi identity
\begin{equation} \label{eq:GenWittJacobi}
    \cs_{a,b,c}\,\big([\sigma(a)\cdot\ps,[
    b\cdot\ps,c\cdot\ps
]_\sigma]_\sigma+\delta\cdot[
a\cdot\ps, [
b\cdot\ps,c\cdot\ps]_\sigma]_\sigma\big)=0.
\end{equation}
\end{theorem}The algebra $\A\cdot\ps$ in the theorem is then a
quasi-Hom-Lie-algebra with $\alpha=\sigma$,
$\beta=\delta$ and $\omega=-\id_{\A\cdot\ps}$.

As example of application of the method in
Theorem \ref{thm:GenWitt} we  review the results
in \cite{LS3,LarsSilvSig-1-sl2} concerned with
this quasi-deformation scheme when applied to the
simple Lie algebra $\sll$.

The Lie algebra $\sll$ can be realized as a
vector space generated by elements $H$, $E$ and
$F$ with the bilinear bracket product defined by
the relations
\begin{align}\label{eq:sl2}
    [ H,E] = 2E, \qquad[ H,F] = -2F,
    \qquad[ E,F] = H.
\end{align}Our basic starting point is the following representation of $\sll$
in terms of first order differential operators
acting on some vector space of functions in the
\text{variable $t$}:
\begin{align*}
    E &\mapsto \partial,\qquad H\mapsto -2t\partial,\qquad F \mapsto
    -t^2\partial.
\end{align*}
To quasi-deform $\sll$ means that we firstly
replace $\partial$ by $\ps$ in this
representation. At our disposal are now the
deformation parameters $\A$ (the "algebra of
functions") and the endomorphism $\sigma$. After
computing the bracket on $\A\cdot\ps$ by
\text{Theorem \ref{thm:GenWitt}} the relations in
the quasi-Lie deformation are obtained by pull
back.

Let $\A$ be a commutative, associative
$\K$-algebra with unity $1$, $t$ an element of
$\A$, and let $\sigma$ denote an $\K$-algebra
endomorphism on $\A$. Also, let $\mathfrak{D}_\sigma(\A)$
denote the linear space of $\sigma$-derivations
on $\A$. Choose an element $\ps$ of
$\mathfrak{D}_\sigma(\A)$ and consider the $\K$-subspace
$\A\cdot\ps$ of elements on the form $a\cdot\ps$
for $a\in\A$. We will usually denote $a\cdot\ps$
simply by $a\ps$. Notice that $\A\cdot\ps$ is a
left $\A$-module, and by \text{Theorem
\ref{thm:GenWitt}} there is a skew-symmetric
algebra structure on $\A\cdot\ps$ given by
\begin{align}
[ a\cdot\ps,
b\cdot\ps]&=\sigma(a)\cdot\ps(b\cdot\ps)-
\sigma(b)\cdot\ps(a\cdot\ps) \notag \\
                           &=(\sigma(a)\ps(b)-\sigma(b)\ps(a))\cdot\ps,\label{eq:bracket}
\end{align}
where $a,b\in\A$. The elements $e:=\ps,
h:=-2t\ps$ and $f:=-t^2\ps$ span an
\text{$\K$-linear} subspace
$\mathcal{S}:=\linSpan_\K\{\ps,-2t\ps,-t^2\ps\}=\linSpan_\K\{e,h,f\}$
of $\A\cdot\ps$. We restrict the multiplication
(\ref{eq:bracket}) to $\mathcal{S}$ without, at
this point, assuming closure. Now, $\ps
(t^2)=\ps(t\cdot
t)=\sigma(t)\ps(t)+\ps(t)t=(\sigma(t)+t)\ps(t).$
Under the natural assumptions $\sigma(1)=1$ and
$\ps(1)=0$ (see \cite{LS2}),
(\ref{eq:bracket}) leads to to
\begin{subequations}
\begin{align}
  & [ h,f]=2\sigma(t)t\ps(t)\ps\label{eq:Shfsimp}\\
  & [ h,e] =2\ps(t)\ps\label{eq:Shesimp}\\
  & [ e,f] =-(\sigma(t)+t)\ps(t)\ps\label{eq:Sefsimp}.
\end{align}
\end{subequations}
\begin{remark}
    Note that when $\sigma=\id$ and $\ps(t)=1$ we retain the
    classical $\sll$ with relations {\rm (\ref{eq:sl2})}.
\end{remark}

In \cite{LS3}, we studied mostly some of the
algebras appearing in the quasi-deformation
scheme in the case when $\A=\K[t]$. This resulted
in new multi-parameter families of sub-quadratic
algebras in particular containing for special
choices of parameters known examples of Lie
algebras, color Lie algebras, and $q$-deformed
Lie algebras. One of such quasi-deformations of
$\sll$ is Jackson $\sll$ corresponding to
discretization of vector fields by Jackson
$q$-derivative and thus to a corresponding
special choice of deformation parameters defined
by the choice of $\A$, $\sigma$ and $\ps$. But we
also have constructed there quasi-Lie
deformations in the case $\A=\K[t]/(t^3)$
yielding new interesting unexpected parametric
families of algebras. In
\cite{LarsSilvSig-1-sl2}, we have constructed
quasi-Lie deformations when $\A=\K[t]/(t^4)$.
This case leads typically to six relations
instead of three which might have been thought as
natural as $\sll$ only has three relations. We
will now review an extension of the construction
to the general class of quasi-Lie deformations
when $\A=\K[t]/(t^N)$ (see also
\cite{LSSAGMFGoet}). We believe that this is a
new, very interesting and rich multi-parameter
family of sub-quadratic algebras that, when
studied from various points of view, may reveal
remarkable properties. It would be of interest to
determine the ring-theoretic properties of these
algebras e.g., for which values of the parameters
are they domains, noetherian, PBW-algebras,
regular \text{etc}, describe their center,
subalgebras, ideals and representations.

Let $\K$ include all $N^{\mathrm{th}}$-roots of
unity and  $\A$ be the algebra $\K[t]/(t^N)$ for
a positive integer $N\geq 2$. This is obviously
an $N$-dimensional $\K$-vector space and a
finitely generated $\K[t]$-module with basis
$\{1,t,\ldots,t^{N-1}\}$. For $i=0,\ldots,{N-1}$,
let $g_i=c_it^i\ps,\ c_i\neq 0$. Put
\begin{equation} \label{eq:pszsimaz}
    \ps(t)=p(t)=\sum_{k=0}^{N-1}p_kt^k, \qquad
    \sigma(t)=\sum_{k=0}^{N-1}q_kt^k,
\end{equation}
considering these as elements in the ring
$\K[t]/(t^N)$. The equalities (\ref{eq:pszsimaz})
have to be compatible with $t^N=0$. This means in
particular that (if $s(t)=(\sigma(t)-q_0)/t)$
\begin{align*}
    \sigma(t^N)&=\sigma(t)^N=(q_0+s(t)t)^N =
    \sum_{\nu=0}^N\binom{N}{\nu}q_0^{\nu}(s(t))^{N-\nu}t^{N-\nu}\\
    &=\sum_{\nu=1}^N\binom{N}{\nu}q_0^{\nu}(s(t))^{N-\nu}t^{N-\nu}
    =q_0\sum_{\nu=1}^N\binom{N}{\nu}q_0^{\nu-1}(s(t))^{N-\nu}t^{N-\nu}=0
\end{align*}
implying (and actually equivalent to) $q_0^N=0$
and hence $q_0=0$. Furthermore,
\begin{align}\label{eq:psz3}
\ps(t^N)&=\sum_{j=0}^{N-1}\sigma(t)^jt^{N-j-1}\ps(t)=
    p(t)\sum_{j=0}^{N-1}s(t)^jt^jt^{N-j-1} \notag \\
    &=p(t)t^{N-1}\sum_{j=0}^{N-1}s(t)^j=
    p_0t^{N-1}\sum_{j=0}^{N-1}s(t)^j=
    p_0\{N\}_{q_1}t^{N-1}=0
\end{align}
where $\{N\}_{q_1} = \sum_{j=0}^{N-1} q_1^{j}$.
It follows thus that
\begin{equation}\label{eq:defcond}
     (1+q_1+q_1^2+\ldots+q_1^{N-1})p_0=0.
\end{equation}
In other words, when $p_0\neq 0$ we generate
deformations at the zeros of the polynomial
$u^{N-1}+\ldots+u^2+u+1$, that is at N'th roots
of unity; whereas if $p_0=0$ then $q_1$ is a true
formal deformation parameter.

As before we make the assumptions that
$\sigma(1)=1$, $\ps(1)=0$ and so relations
(\ref{eq:Shfsimp}), (\ref{eq:Shesimp}) and
(\ref{eq:Sefsimp}) still hold. Moreover, since
for $k\ge 0$ we have $$
    \ps(t^{k+1})=\sum_{j=0}^k\sigma(t)^jt^{k-j}\ps(t)=
    p(t)t^k\sum_{j=0}^ks(t)^j.
$$ Using \eqref{eq:bracket}, we get
\begin{eqnarray}
    [ g_i,g_j] =c_ic_j[ t^i\ps,
    t^j\ps] = c_ic_j[\sigma(t^i)\ps(t^j)-\sigma(t^j)\ps(t^i)]\ps
    \nonumber\\
    =c_ic_j[\sigma(t)^i\ps(t^j)-\sigma(t)^j\ps(t^i)]
    \ps =c_ic_j[\ps(t^j)-\sigma(t)^{j-i}\ps(t^i)]
    \sigma(t)^i\ps.
    \label{eq:Shg}
\end{eqnarray}
By (\ref{eq:GenWittProdDef}) the bracket can be
computed abstractly on generators $g_i,g_j$ as
\begin{align}
    \label{eq:brack_hf}
  [ g_i,g_j]&=
  [ c_it^i\ps,c_jt^j\ps] = c_ic_j[(\sigma(t^i)\ps)\circ(t^j\ps)-(\sigma(t^j)\ps)\circ(t^i\ps)] \notag \\
         &= c_i\sigma(t)^i\ps\circ g_j-
         c_j\sigma(t)^j\ps\circ g_i.
    \end{align}
 Expanding according to the multinomial formula
 $s(t)^k = (q_1+q_2t+\ldots+q_{N-1}t^{N-2})^k$ and
 $\sigma(t^k)= \sigma(t)^k = t^k s(t)^k=
  (q_1+q_2t+\ldots+q_{N-1}t^{N-2})^k$
 we obtain
 \begin{align*}
  &[ g_i,g_j] =
  c_i(s(t)^it^i\ps)\circ g_j-c_j(s(t)^jt^j\ps)
  \circ g_i \\
         &= c_ii!\Big(\sum_{\substack{i_1,\ldots,i_{N-1}\ge 0 \\ i_1+\ldots+i_{N-1}=i \\
         i_2+2i_3+\ldots+(N-2)i_{N-1}<N-i}}
         \ \frac{q_1^{i_1}\cdots q_{N-1}^{i_{N-1}}}{i_1!\cdots i_{N-1}!}\
         \frac{g_{i+i_2+2i_3+\ldots+(N-2)i_{N-1}}\ g_j}{c_{i+i_2+2i_3+\ldots+(N-2)i_{N-1}}}\Big) \\
         &-c_jj!\Big(\sum_{\substack{j_1,\ldots,j_{N-1}\ge 0 \\ j_1+\ldots+j_{N-1}=j \\
         j_2+2j_3+\ldots+(N-2)j_{N-1}<N-j}}
         \ \frac{q_1^{j_1}\cdots q_{N-1}^{j_{N-1}}}{j_1!\cdots j_{N-1}!}\
         \frac{g_{j+j_2+2j_3+\ldots+(N-2)j_{N-1}}\ g_i}{c_{j+j_2+2j_3+\ldots+(N-2)j_{N-1}}}\Big).
    \end{align*}
The bracket is closed on linear span of $g_i$'s
as for $N-1\geq i,j \geq 0$, by
\eqref{eq:bracket}, we get
\begin{align}
    [ g_i,g_j]&=
c_ic_j[\ps(t^j)-\sigma(t)^{j-i}\ps(t^i)]
\sigma(t)^i\ps \notag \\
    &=c_ic_j\sum_{k=0}^{|j-i|-1} \sign(j-i)
    \sum_{\substack{k_1,k_2,\ldots,k_{N-1}\ge 0 \\
    k_1+k_2+\ldots+k_{N-1}=k+\min\{i,j\} \\
         k_2+2k_3+\ldots+(N-2)k_{N-1}<N}}
         \frac{(k+\min\{i,j\})!}{k_1!k_2!\cdots
         k_{N-1}!}\ \notag\\
         &\times q_1^{k_1}q_2^{k_2}\ldots
         q_{N-1}^{k_{N-1}}
         t^{k_2+2k_3+\ldots+(N-2)k_{N-1}}
         \sum_{l=0}^{N-1}p_lt^{i+j+l-1}\ps. \notag
\\
    &=c_ic_j\sum_{l=0}^{N-1}p_l\sum_{k=0}^{|j-i|-1}
    \sign(j-i) \sum_{\substack{k_1,k_2,\ldots,k_{N-1}\ge 0 \\ k_1+k_2+\ldots+k_{N-1}=i+k \\
         k_2+2k_3+\ldots+(N-2)k_{N-1}\le N-i-j-l}}
         \frac{(k+\min\{i,j\})!}{k_1!k_2!\cdots k_{N-1}!}\ \\
         &\times q_1^{k_1}q_2^{k_2}\ldots q_{N-1}^{k_{N-1}}\
         \frac{g_{i+j+l-1+k_2+2k_3+\ldots+(N-2)k_{N-1}}}{c_{i+j+l-1+k_2+2k_3+\ldots+(N-2)k_{N-1}}}.
\end{align}
where $\sign(x)= -1$ if $x<0$, $\sign(x)= 0$ if
$x=0$ and $\sign(x)= 1$ if $x>0$.

Now we turn to the other example, quasi-Lie
deformation of Witt algebra using
$\sigma$-derivations on ${\mathbb C}[t,t^{-1}]$,
where $\sigma$ is arbitrary endomorphism of the
algebra ${\mathbb C}[t^{\pm 1}]$. These algebras
were constructed by Hartwig, Larsson and
Silvestrov in \cite[Theorem 31]{HLS} as  an
outcome of application of Theorem
\ref{thm:GenWitt}. Let $\A=\K[t,t^{-1}]$, the
algebra of Laurent polynomials. Any endomorphism
of $\A=\K[t,t^{-1}]$ is uniquely determined by
its action on the generator $t$. So, assume that
$\sigma(t)=p(t)\neq 0 \in\A$. Note that
$\sigma(1)=1$ and $\sigma(t^{-1})=\sigma(t)^{-1}$
since $\A$ has no zero-divisors. Hence, since
$\sigma(t)$ is invertible, $\sigma(t)=p(t)=qt^s$,
for some $q\in\K\setminus\{0\}$ and $s\in\Z$.
Since $\A=\K[t,t^{-1}]$ is a unique factorization
domain, if $\sigma:\A\to\A$ is a homomorphism
different from the identity map $\id$ (either $q\neq
1$ or $s\neq 1$), then $\mathfrak{D}_\sigma(\A)=\A\cdot
D$, where $D=\frac{\id-\sigma}{g}$ and
$g=\GCD\big((\id-\sigma)(\A)\big)$, and the outer
twisting multiplier in 6 term Jacobi identity
\eqref{eq:GenWittJacobi} can be computed then as
$ \delta=\sigma(g)/g$ (see \cite{HLS}). It
suffices to compute a greatest common divisor of
$(\id-\sigma)(\A)$ on the generator $t$ since
$\sigma(t^{-1})$ is determined by $\sigma(t)$ and
any $\gcd$ is only determined up to a multiple of
an invertible element. Thus $g=\eta^{-1}
t^{k-1}(t-qt^{s})$ is a perfectly general $\gcd$
and $D=\eta t^{-k}\frac{\id-\sigma}{1-qt^{s-1}} $
is a generator for $\mathfrak{D}_\sigma(A)$ as a left
$\A$-module. The $\sigma$-derivations on
$\K[t,t^{-1}]$ are on the form $f(t)\cdot D$ for
$f\in\K[t,t^{-1}]$ and so, given that $t^{\Z}$ is
a linear basis of $\K[t,t^{-1}]$ (over $\K$),
$-t^{\Z}\cdot D$ is a linear basis (over $\K$
again) for $\mathfrak{D}_\sigma(\K[t,t^{-1}])$. By
Theorem \ref{thm:GenWitt}, the linear space
    $\mathfrak{D}_\sigma(\A) = \bigoplus_{n\in{\mathbb Z}}
    \K \cdot d_n,$
    where $d_n=-t^nD$, $n\in \mathbb{Z}$
can be equipped with the bilinear bracket product
defined on generators (by
\eqref{eq:GenWittProdDef}) as $ [
d_n,d_m] _\sigma =
q^nd_{ns}d_m-q^md_{ms}d_n $ and satisfying
defining commutation relations
$$\begin{array}{l}
{[ d_n,d_m]
_\sigma}=\alpha{\rm{sign}}(n-m) \sum_{l={\rm
min}(n,m)}^{{\rm{max}}(n,m)-1}q^{n+m-1-l}
            d_{s(n+m-1)-(k-1)-l(s-1)} \\
            \qquad\textrm{for } n,m\geq 0;\\
     {[ d_n,d_m]}_\sigma=\alpha\Big(
     \sum_{l=0}^{-m-1}q^{n+m+l}d_{(m+l)(s-1)+ns+m-k}
+\sum_{l=0}^{n-1}q^{m+l}d_{(s-1)l+n+ms-k}\Big)\\
           \qquad\textrm{for } n\geq 0, m<0;
\\
{[ d_n,d_m]}_\sigma=-\alpha\Big(
            \sum_{l_1=0}^{m-1}q^{n+l_1}
        d_{(s-1)l_1+m+ns-k}
       +\sum_{l_2=0}^{-n-1}q^{m+n+l_2}d_{(n+l_2)(s-1)+n+ms-k}\Big)\\
            \qquad\textrm{for } m\geq 0, n<0;  \\
           {[ d_n,d_m]}_\sigma=\alpha{\rm{sign}}(n-m)
           \sum_{l={\rm min}(-n,-m)}^{{\rm max}(-n,-m)-1}q^{n+m+l}d_{(m+n)s+(s-1)l-k}\\
            \qquad\textrm{for } n,m<0,
\end{array}$$
skew-symmetry
    $[d_n,d_m]_\sigma=-[d_m,d_n]_\sigma$ and
    a twisted Jacobi identity
 \begin{equation*}
        \circlearrowleft_{n,m,l}\,\Big (q^n\big [ d_{ns},[ d_m,d_l] _\sigma\big
 ] _\sigma+\delta\big
 [
 d_n[ d_m,d_l] _\sigma\big ] _\sigma
 \Big )=0,
  \end{equation*}
where
$\delta=q^kt^{k(s-1)}\sum_{r=0}^{s-1}(qt^{s-1})^r
\in \K[t,t^{-1}]$. We get a family of
quasi-Hom-Lie algebras deforming the Witt
algebra. Here the twisting map $\alpha$ is a
linear map on $\mathfrak{D}_\sigma(\A)$ sending $d_n$ to
$q^nd_{ns}$, the twisting map $\beta$ is a linear
map on $\mathfrak{D}_\sigma(\A)$ acting as the
multiplication by $\delta$ from the left, and
$\omega = \theta = -\id$

A Hom-Lie algebra is obtained when $\delta \in
\K$, and this can be achieved only when $s=1$,
that is, when $\sigma(t)=qt$ (See Example
\ref{ex:qWitt}). If $s=1$ then
$[d_n,d_m]_q=q^nd_nd_m-q^md_md_n$ which is the
bracket for the usual $q$-Witt algebra associated
to discretizations by Jackson $q$-derivative
\cite[Theorem 27]{HLS}), reducing further to the
usual commutator for Witt Lie algebra if $s=1$
and $q=1$. For further analysis of these
quasi-Lie algebras quasi-deformations of Witt algebra we
refer to \cite{RichardSil,RS1AGMFLund}.

\section{Hom-Algebra  and Hom-Coalgebra structures}
\label{sec:HomalgHomcoalg} The notions of
Hom-associative, Hom-Leibniz, and
Hom-Lie-admissible algebraic structures was
introduced in \cite{MS}, generalizing the well
known associative, Leibniz and Lie-admissible
algebras. The Hom-Poisson algebra was introduced in \cite{HomDeform}, it is suitable for deformation theory of commutative Hom-associative algebras.

By dualization of Hom-associative
algebra we define in the sequel  the Hom-coassociative coalgebra
structure.

\begin{definition}
A \emph{Hom-associative algebra} is a triple
$(V,\mu,\alpha)$ where $V$ is a $\K$-linear
space, $\mu : V\otimes V \rightarrow V$ is a
bilinear multiplication and $\alpha:V \rightarrow
V$ is a $\K$-linear space homomorphism satisfying
the Hom-associativity condition
\begin{equation}\label{Hom-ass}
\mu(\alpha(x)\otimes \mu (y\otimes z))= \mu (\mu
(x\otimes y)\otimes \alpha (z)).
\end{equation}
\end{definition}
The Hom-associativity condition \eqref{Hom-ass}
may be expressed by the following commutative
diagram.
$$
\begin{array}{ccc}
V\otimes V\otimes V & \stackrel{\mu \otimes
\alpha}{\longrightarrow } & V\otimes
V \\
\ \quad \left\downarrow ^{\alpha\otimes \mu
}\right. &  & \quad \left\downarrow
^\mu \right. \\
V\otimes V & \stackrel{\mu }{\longrightarrow } &
V
\end{array}
$$

The Hom-associative algebra is said to be \emph{unital} if there
exists a homomorphism $\eta :\K\rightarrow V$
such that the following diagrams are commutative

$$\begin{array}{ccccc} \K \otimes V &
\stackrel{\eta \otimes id}{\longrightarrow } &
V\otimes V &
\stackrel{id\otimes \eta }{\longleftarrow } & V\otimes \K \\
& \searrow ^{\cong } & \quad \left\downarrow ^\mu
\right. & \swarrow ^{\cong
} &  \\
&  & V &  &
\end{array}
$$

Let $\left( V,\mu ,\alpha \right) $ and $\left(
V^{\prime },\mu ^{\prime },\alpha^{\prime
}\right) $ be two Hom-associative algebras. A
linear map $f\ :V\rightarrow V^{\prime }$ is said
to be a
\emph{morphism of Hom-associative algebras} if%
$$  \mu ^{\prime }\circ (f\otimes f)=f\circ \mu
\text{,} \qquad f\circ \alpha=\alpha^{\prime
}\circ f \qquad
$$
and  $f\circ \eta =\eta ^{\prime }$ if the
Hom-algebras are unital with units $\eta$ and
$\eta^{\prime}$.


The tensor product of two Hom-associative
algebras $\left( V_1,\mu _1,\alpha_1 \right) $
and $\left( V_2,\mu_2,\alpha_2  \right) $ is
defined in an obvious way as the Hom-associative
algebra $\left( V_1\otimes V_2,\mu _1 \otimes
\mu_2,\alpha_1 \otimes \alpha_2  \right) $. If
$\eta_1$ and $\eta_2$ are the units of
these Hom-associative algebras, then the tensor
product is also unital with the unit
$\eta_1 \otimes\eta_2$.

\begin{example}\label{example1ass}
Let $\{x_1,x_2,x_3\}$  be a basis of a $3$-dimensional linear space
$V$ over $\K$. The following multiplication $\mu$ and linear map
$\alpha$ on $V$ define a Hom-associative algebra over $\K^3${\rm :}
$$
\begin{array}{ll}
\begin{array}{lll}
 \mu ( x_1,x_1)&=& a x_1, \ \\
\mu ( x_1,x_2)&=&\mu ( x_2,x_1)=a x_2,\\
\mu ( x_1,x_3)&=&\mu ( x_3,x_1)=b x_3,\\
 \end{array}
 & \quad
 \begin{array}{lll}
\mu ( x_2,x_2)&=& a x_2, \ \\
\mu ( x_2, x_3)&=& b x_3, \ \\
\mu ( x_3,x_2)&=& \mu ( x_3,x_3)=0.
  \end{array}
\end{array}
$$

$$  \alpha (x_1)= a x_1, \quad
 \alpha (x_2) =a x_2 , \quad
   \alpha (x_3)=b x_3
$$
where $a,b$ are parameters in $\K$. This algebra
is not associative when $a\neq b$ and $b\neq 0$,
since
$$\mu (\mu (x_1,x_1),x_3))- \mu ( x_1,\mu
(x_1,x_3))=(a-b)b x_3.$$

\end{example}

In \cite{Yau:HomolHomLie}, D. Yau shows that one can construct a Hom-associative algebra starting from an associative algebra and an algebra endomorphism. Therefore using the following theorem, one can provide examples of Hom-associative algebras.
\begin{theorem}[\cite{Yau:HomolHomLie}]\label{thmYauConstrHomAss}
Let $(V,\mu)$ be an associative algebra and let $\alpha : V\rightarrow V$ be an algebra endomorphism. Then $(V,\mu_\alpha,\alpha)$, where $\mu_\alpha=\alpha\circ\mu$,  is a Hom-associative algebra.

Moreover, suppose that  $(V',\mu')$ is another associative algebra and  $\alpha ' : V'\rightarrow V'$ is an algebra endomorphism. If $f:V\rightarrow V'$ is an algebra morphism that satisfies $f\circ\alpha=\alpha'\circ f$ then
$$f:(V,\mu_\alpha,\alpha)\longrightarrow (V',\mu'_{\alpha '},\alpha ')
$$
is a morphism of Hom-associative algebras.
\end{theorem}
\begin{proof}
We show that $(V,\mu_\alpha,\alpha)$ satisfies the Hom-associativity. Indeed
\begin{align*}
\mu_\alpha(\alpha(x)\otimes \mu_\alpha (y\otimes z))&=\alpha(\mu(\alpha(x)\otimes \alpha(\mu (y\otimes z)))\\
&=\alpha(\mu(\alpha(x)\otimes \mu ( \alpha (y)\otimes  \alpha (z)))\\
&=\alpha(\mu( \mu (\alpha(x)\otimes \alpha (y))\otimes  \alpha (z)))\\
&=\alpha(\mu( \alpha(\mu (x\otimes y))\otimes  \alpha (z)))\\
&= \mu_\alpha (\mu_\alpha
(x\otimes y)\otimes \alpha (z))
\end{align*}
The second assertion is proved similarly.
\end{proof}

The free Hom-associative algebra was constructed  in \cite{Yau:EnvLieAlg}. In the following,
we give  some identities satisfied by a
Hom-associative algebra $(V,\cdot,\alpha)$ and a pentagonal diagram under certain assumptions.
\begin{lemma}
Let  $(V,\cdot,\alpha)$ be a Hom-associative algebra. Then, for  $n\geq 2$, $ x_0,x_1,\cdots,x_n\in V$ the following is true
\begin{eqnarray*}
\alpha^{n-1}(x_0)\cdot (\alpha^{n-2} (x_1)\cdot (\alpha^{n-3} (x_2)\cdots( \alpha(x_{n-2})\cdot (x_{n-1}\cdot x_{n}))\cdots))=\quad \quad\quad \\
((\cdots((x_0\cdot x_1)\cdot\alpha (x_2))\cdots\alpha^{n-3} (x_{n-2}))\cdot\alpha^{n-2} (x_{n-1}))\cdot\alpha^{n-1} (x_n)
\end{eqnarray*}
\end{lemma}
\begin{proof}
The case $n=1$ corresponds to the associativity condition. For $n=2$, we have
\begin{align*}
\alpha^2(x_0)\cdot (\alpha (x_1)\cdot( x_2\cdot x_3))&=\alpha^2(x_0)\cdot ((x_1\cdot x_2)\cdot \alpha (x_3))\\
&=(\alpha(x_0)\cdot (x_1\cdot x_2))\cdot \alpha^2 (x_3)\\
&=((x_0\cdot x_1)\cdot \alpha(x_2))\cdot \alpha^2 (x_3)\\
\end{align*}
Similarly, one may obtain the complete proof by induction on $n$.
\end{proof}
\begin{remark}
If the homomorphism $\alpha$ is invertible, the previous identity is equivalent to
\begin{eqnarray*}
x_0\cdot ( x_1\cdot (x_2\cdots( x_{n-2}\cdot (x_{n-1}\cdot x_{n}))\cdots))=\quad \quad\quad \quad \quad\quad\\
((\cdots((\alpha^{1-n} (x_0)\cdot \alpha^{2-n} (x_1))\cdot\alpha^{4-n} (x_2))\cdots\alpha^{n-4} (x_{n-2}))\cdot\alpha^{n-2} (x_{n-1}))\cdot\alpha^{n-1} (x_n)
\end{eqnarray*}
In particular,
\begin{align*}
x_0\cdot ( x_1\cdot x_2))&=(\alpha^{-1} (x_0)\cdot x_1)\cdot\alpha (x_2))\\
x_0\cdot ( x_1\cdot( x_2\cdot x_3))&=((\alpha^{-2} (x_0)\cdot \alpha^{-1}(x_1))\cdot\alpha (x_2))\cdot\alpha^{2}(x_3)\\
x_0\cdot ( x_1\cdot( x_2\cdot( x_3\cdot x_4)))&=((\alpha^{-3} (x_0)\cdot \alpha^{-2}(x_1))\cdot x_2)\cdot\alpha^{2}(x_3))\cdot\alpha^{3}(x_4)
\end{align*}
\end{remark}

To construct a pentagonal diagram, similar to Mac Lane pentagon, we  start with $\alpha^2(x_0)\cdot (\alpha (x_1)\cdot( x_2\cdot x_3))$ and by  Hom-associativity we have
\begin{align*}
\alpha^2(x_0)\cdot (\alpha (x_1)\cdot( x_2\cdot x_3))&=\alpha^2(x_0)\cdot ((x_1\cdot x_2)\cdot \alpha (x_3))\\
&=(\alpha(x_0)\cdot (x_1\cdot x_2))\cdot \alpha^2 (x_3)\\
&=((x_0\cdot x_1)\cdot \alpha(x_2))\cdot \alpha^2 (x_3)\\
&=( \alpha(x_0\cdot x_1)\cdot( \alpha(x_2)\cdot \alpha (x_3))\\
\end{align*}
In another hand
\begin{align*}
\alpha^2(x_0)\cdot (\alpha (x_1)\cdot( x_2\cdot x_3))&=(\alpha(x_0)\cdot \alpha(x_1))\cdot \alpha ( x_2\cdot x_3)
\end{align*}
Therefore
$$\alpha (x_0\cdot x_1)\cdot (\alpha (x_2)\cdot \alpha(x_3))  =  (\alpha(x_0)\cdot
\alpha(x_1))\cdot \alpha (x_2 \cdot x_3)
$$
Then, if $\alpha$ is an algebra  homomorphism, that is $\alpha(x \cdot y)=\alpha(x)\cdot \alpha(y)$, the previous identity is satisfied and therefore we obtain the following commutative pentagonal diagram

$$
\begin{array}{ccccc}
\ &\alpha^2(x_0)\cdot (\alpha (x_1)\cdot( x_2\cdot x_3)) &\ & \ \\
  \swarrow &  \ &  \searrow \\
\alpha^2(x_0)\cdot ((x_1\cdot x_2)\cdot \alpha (x_3))  & \   & (\alpha(x_0)\cdot \alpha(x_1))\cdot( \alpha ( x_2)\cdot \alpha( x_3) )\\
 \downarrow& \   &\uparrow\\
(\alpha(x_0)\cdot (x_1\cdot x_2))\cdot \alpha^2 (x_3) &\longrightarrow& ((x_0\cdot x_1)\cdot \alpha(x_2))\cdot \alpha^2 (x_3) & \
\end{array}
$$
\vspace{0.5cm}

The Hom-Lie algebras which is a particular case of quasi-Lie and quasi-Hom-Lie algebras is defined as follows.
\begin{definition} \label{def:HomLie}
A \emph{Hom-Lie algebra} is a triple $(V, [\cdot,
\cdot], \alpha)$ consisting of
 a linear space $V$, bilinear map $[\cdot, \cdot]: V\times V \rightarrow V$ and
 a linear space homomorphism $\alpha: V \rightarrow V$
 satisfying
$$\begin{array}{c} [x,y]=-[y,x] \quad {\text{(skew-symmetry)}} \\{}
\circlearrowleft_{x,y,z}{[\alpha(x),[y,z]]}=0
\quad {\text{(Hom-Jacobi condition)}}
\end{array}$$
for all $x, y, z$ from $V$.
\end{definition}

\begin{example}
Let $\{x_1,x_2,x_3\}$  be a basis of a $3$-dimensional linear space
$V$ over $\K$. The following bracket and   linear map $\alpha$ on
$V$ define a Hom-Lie algebra over $\K^3${\rm :}
$$
\begin{array}{cc}
\begin{array}{ccc}
 [ x_1, x_2 ] &= &a x_1 +b x_3 \\ {}
 [x_1, x_3 ]&=& c x_2  \\ {}
 [ x_2,x_3 ] & = & d x_1+2 a x_3,
 \end{array}
 & \quad

  \begin{array}{ccc}
  \alpha (x_1)&=&x_1 \\
 \alpha (x_2)&=&2 x_2 \\
   \alpha (x_3)&=&2 x_3
  \end{array}

\end{array}
$$
with $[ x_2, x_1 ]$, $[x_3, x_1 ]$ and  $[
x_3,x_2 ]$ defined via skewsymmetry. It is not a
Lie algebra if and only if $a\neq0$ and $c\neq0$,
since
$$[x_1,[x_2,x_3]]+[x_3,[x_1,x_2]]
+[x_2,[x_3,x_1]]= a c x_2.$$
\end{example}

We call a triple $\left( V,\mu ,\alpha \right) $
\emph{Hom-Lie admissible algebra} if the commutator
defined for $x,y\in V  $ by $[ x,y ]=\mu
(x,y)-\mu (y,x ) $ defines a Hom-Lie algebra $(V,
[\cdot, \cdot], \alpha)$. The Hom-Lie admissible
algebras were studied in \cite{MS}, where we have
shown among other things, that the
Hom-associative algebras are Hom-Lie admissible,
thus generalizing the well known fact that the
commutator on associative algebra defines a Lie
algebra structure.

\begin{proposition}
To any Hom-associative algebra defined by the multiplication $\mu$
and a homomorphism $\alpha$ over a $\K$-linear space $V$, one may
associate a  Hom-Lie algebra defined for all $x,y \in V$ by the
bracket
$
[ x,y ]=\mu (x,y)-\mu (y,x ).
$
\end{proposition}

Thus, we have a functor from the category $HomAss$ of Hom-associative algebras into a category $HomLie$ of Hom-Lie algebras.  Its left adjoint functor was constructed by Yau in \cite{Yau:EnvLieAlg}. It corresponds to the enveloping algebra of a Hom-Lie algebra. The construction makes use of the combinatorial objects of weighted binary trees.

Removing the skew-symmetry and rearranging the
Hom-Jacobi condition we get similarly the class
of Hom-Leibniz algebra.
\begin{definition}
A \emph{Hom-Leibniz algebra} is a triple $(V,
[\cdot, \cdot], \alpha)$ consisting of a linear
space $V$, bilinear map $[\cdot, \cdot]: V\times
V \rightarrow V$ and a homomorphism $\alpha: V
\rightarrow V$  satisfying
\begin{equation} \label{Leibnizalgident}
 [[x,y],\alpha(z)]=[[x,z],\alpha (y)]+
 [\alpha(x),[y,z]].
\end{equation}
\end{definition}
Note that if a Hom-Leibniz algebra is
skew-symmetric then it is a Hom-Lie algebra.

We introduce in the following definition
Hom-Poisson structure involved naturally in the
deformation theory of Hom-Lie algebras
\cite{HomDeform}.

\begin{definition}
A \emph{Hom-Poisson algebra} is a quadruple $(V,\mu, \{\cdot,
\cdot\}, \alpha)$ consisting of
 a linear space $V$, bilinear maps $\mu: V\times V \rightarrow V$ and
  $\{\cdot, \cdot\}: V\times V \rightarrow V$, and
 a linear space homomorphism $\alpha: V \rightarrow V$
 satisfying
 \begin{enumerate}
\item $(V,\mu, \alpha)$ is a commutative Hom-associative algebra,
\item $(V, \{\cdot,
\cdot\}, \alpha)$ is a Hom-Lie algebra,
\item
for all $x, y, z$ in $V$,
\begin{equation}\label{CompatibiltyPoisson}
\{\alpha (x) , \mu (y,z)\}=\mu (\alpha (y),
\{x,z\})+ \mu (\alpha (z), \{x,y\}).
\end{equation}
\end{enumerate}
\end{definition}
The condition \eqref{CompatibiltyPoisson}
expresses the compatibility between the
multiplication and the Poisson bracket. It can be
 reformulated equivalently as
\begin{equation}\label{CompatibiltyPoissonLeibform}
\{\mu(x,y),\alpha (z) \}=\mu (\{x,z\},\alpha
(y))+\mu (\alpha (x), \{y,z\})
\end{equation}
for all $x, y, z$ in $V$. Note that in this form
it means that $ad_z (\cdot) = \{\cdot,z\}$ is a sort  of generalization of
derivation of associative  algebra defined by  $\mu$, and
also it resembles the identity
\eqref{Leibnizalgident} in the definition for
Leibniz algebra.

It is good place now to give some examples of
Hom-Lie algebras.
It is well known that
$\sll$ is a rigid Lie algebra, that is every
formal deformation is equivalent to trivial
deformation. In the following, we provide
examples of Hom-Lie algebras as  deformations of
the classical Lie algebra $\sll$ defined by $[
x_1,x_2] = 2x_2,\  [ x_1,x_3] = -2x_3 , \  [
x_2,x_3] =x_1$ and  $q$-deformed  Witt algebra.

\begin{example}
We construct  Hom-Lie infinitesimal formal deformations of $\sll$ which are not Lie algebras.  We consider the  $3$-dimensional Hom-Lie algebras with the skew-symmetric bracket $[~,~]_t$ and linear map $\alpha_t$  defined  as follows
$$
\begin{array}{cc}
\begin{array}{ccc}
 [ x_1, x_2 ] _t&= &a_1 t x_1+(2- a_2 t) x_2  \\ {}
 [x_1, x_3 ]_t&=& a_3 t x_1+a_4 t x_2+(-2+ a_2 t) x_3, \\ {}
 [ x_2,x_3 ] _t& = & (1-\frac{a_2}{2} t) x_1,
 \end{array}
 & \ \
(\alpha_t)= \left(
  \begin{array}{ccc}
 1+ b_1 t & \frac{a_1}{2}t& \frac{b_2-a_3}{2}t \\
 b_2 t &1- \frac{a_2 }{2}t&-\frac{a_4}{2}t \\
    0 & 0 &1- \frac{a_2 }{2}
  \end{array}
\right)
\end{array}
$$
where $ a_1,a_2,a_3,a_4,b_1,b_2$ are parameters in $\K$.

These Hom-Lie algebras become  Lie algebras for all $t$ if and only if $a_1=0$ and $a_3=0$, as follows from
$$[x_1,[x_2,x_3]]+[x_3,[x_1,x_2]]+[x_2,[x_3,x_1]]=
(2a_3t -(a_2a_3 +a_1a_4)t^2)x_2+(2a_1t-a_1a_2 t^2)x_3=0.
$$

\end{example}
\begin{example}[\textbf{Jackson $\sll$}]

We consider the Hom-Lie algebra Jackson $\sll$
which is a Hom-Lie deformation of the classical Lie algebra  $\sll$. The Jackson $\sll$ is
related to Jackson derivations.  As linear space, it is generated by
$x_1,x_2,x_3$ with the skew-symmetric brackets defined by
$$[ x_1,x_2]_t = 2x_2,\ \ [ x_1,x_3]_t = -2x_3-2tx_3 , \ \ [ x_2,x_3]_t =x_1+\frac{t}{2}x_1.$$

The linear map $\alpha_t$ is defined by
$$\alpha_t(x_1)= x_1, \quad\alpha_t(x_2)=\frac{2+t}{2(1+t)}x_2=x_2+\sum_{k=0}^{\infty}{\frac{(-1)^k}{2}t^k\ } x_2,
 \quad \alpha_t(x_3)=x_3+\frac{t}{2} \, x_3.$$
The Hom-Jacobi identity is proved as follows. It is enough to
consider it on $x_2$, $x_3$ and $x_1$:
\begin{align*}
[ \alpha_t(x_2),&[ x_3,x_1]_t]_t+[ \alpha_t(x_3),[ x_1,x_2]_t]_t+[
\alpha_t(x_1),[ x_2,x_3]_t]_t=\\
&=(2+t)[ x_2,x_3]_t+(2+t)[ x_3,x_2]_t+\frac{(2+t)}{2}[ x_1,x_1]_t =0.
\end{align*}

%
\end{example}

\begin{example} \label{ex:qWitt}
Let $\A$ be the unique factorization domain
$\K[z,z^{-1}]$, the Laurent polynomials in $z$
over the field $\K$. Then the space ${\mathcal
D}_\sigma(\A)$ can be generated by a single
element $D$ as a left $\A$-module, that is,
${\mathcal D}_\sigma(\A)=\A\cdot D$ (This is a special case of
\cite[Theorem 4]{HLS}.) When $\sigma(z)=qz$ with $q\neq0$
and $q\neq 1$, one can take $D$ as $z$ times the
Jackson $q$-derivative
$$D=\frac{id-\sigma}{1-q}\quad :\quad
f(z)\mapsto\frac{f(z)-f(qz)}{1-q}.$$
The $\K$-linear space
${\mathcal D}_\sigma (\A) = \bigoplus_{n\in \mathbb{Z}}
\K\cdot x_n,$ with
$x_n=-z^nD$ can be equipped with the skew-symmetric
bracket $[ \cdot,\cdot ]_\sigma$
defined on generators  as
\begin{equation} \label{qWittrel} [x_n,x_m]=
q^nx_{n}x_m-q^mx_{m}x_n =[x_n,x_m]
=(\{n\}_q-\{m\}_q)x_{n+m},
\end{equation}
where $\{n\}_q=(q^n-1)/(q-1)$ for $q\neq 1$
and $\{n\}_1=n.$
This bracket is skew-symmetric and
satisfies the $\sigma$-deformed Jacobi identity
\begin{equation} \label{qWittJacobi}
(q^n+1)[ x_n, [ x_l, x_m]]+ (q^l+1)[ x_l,  [ x_m, x_n]_\sigma]
+(q^m+1)[ x_m, [ x_n,x_l]]=0.
\end{equation}
Therefore,
$(V,[\cdot,\cdot],\alpha) =
(\bigoplus_{i\in \mathbb{Z}} \K
x_n, [\cdot,\cdot],\alpha)$
with the bilinear bracket defined on generators as
$[x_n,x_m] = (\{n\}_q-\{m\}_q)x_{n+m}$ and
the linear twisting map $\alpha:V\rightarrow V$
acting on generators
as $\alpha (x_n) =  (q^n+1) x_n$ is an example of
Hom-Lie algebra \cite{LS2,MS} (see Definition \ref{def:HomLie} in this paper).
Obviously this Hom-Lie algebra can be viewed
as a $q$-deformation of Witt algebra in the sense that
for $q = 1$ indeed one recovers the bracket and
the commutation
relations for generators of the Witt algebra.
The definition of its generators using first
order differential operators is recovered if one assumes that $D= z\frac{d}{dz}$
for $q=1$ as one would expect from
passing to a limit
in the definition of the operator $D$.

It can be also shown that there is a central
extension $Vir_q$ of this deformation in the category of hom-Lie
algebras with the deformed Jacobi identity \eqref{qWittJacobi}, which is a natural $q$-deformation
of the Virasoro algebra. The algebra $Vir_q$ is
spanned by elements
$\{x_n\,|\,n\in\mathbb{Z}\}\cup\{{\bf c}\}$ where
${\bf c}$ is central with respect to Hom-Lie bracket,
i.e.,
$[ Vir_q,{\bf c}]=
[{\bf c}, Vir_q]=0$. The
bracket of $x_n$ and $x_m$ is computed according to
$$[ x_n,x_m]=(\{n\}_q-\{m\}_q)x_{m+n}+
\delta_{n+m,0} \frac{q^{-n}}{6(1+q^n)}\{n-1\}_q\{n\}_q\{n+1\}_q{\bf c}.$$
Note that when $q=1$ we retain the
classical Virasoro algebra
$$[ x_n,x_m]=(n-m)x_{m+n}+
\delta_{n+m,0}
\frac{1}{12}\{n-1\}\{n\}\{n+1\} {\bf c}.$$
from conformal field and string theories.
Note also that when specializing ${\bf c}$ to
zero, or equivalently rescaling ${\bf c}$
by extra parameter and then letting the
parameter degenerate to zero,
one recovers the $q$-deformed Witt
algebra. Because of this the Witt
algebra is called also,
primarily in the physics literature,
a centerless Virasoro algebra. In a similar way
the $q$-deformed Witt algebra could be called a
centerless $q$-deformed Virasoro algebra but of
course with the word "central" used in terms of
the Hom-Lie algebra bracket.

In \cite{HomDeform} it was shown how the $q$-deformed Witt algebras
can be viewed in the framework of deformation theory of Hom-Lie
algebras.
\end{example}

\begin{remark}
In \cite{HLS}, the Example \ref{ex:qWitt} has been presented as the simplest example of application of the general method for quasi-Lie quasi-deformation and discrete modification of Lie algebras of vector fields based on general $\sigma$-derivations.
To our knowledge, analogous examples of $q$-deformed Witt algebras  and
$q$-Virasoro algebras associated to ordinary $q$-derivatives
have been constructed for the first time  in 1990 in \cite{AizawaSaito,ChaiElinPop,ChaiIsKuLuk,CurtrZachos1}
where also the $q$-deformed Jacobi identities for these algebras have been discovered.
The earliest paper we know, approaching such example of $q$-deformed Witt algebra  and
of the associated with it $q$-Virasoro algebra in the way closest to Hom-Lie algebras approach
by systematically using the $q$-deformed Jacobi identity as the identity for a non-associative algebra for obtaining central extension, was the paper by Aizawa and Sato \cite{AizawaSaito} from 1990-1991.
In that paper, the authors achieve an important insight by defining a class of non-associative algebras basically almost as general as Hom-Lie algebras and use it systematically
for construction of the $q$-Virasoro algebra
as central extension in terms of this class of non-associative algebras.
The only minor difference of their definition from Hom-Lie algebras (see Definition \ref{def:HomLie})
is that they put extra condition that in addition to skew-symmetry and a Hom-Jacobi identity
there must be a "special limit"
so that in the limit a Lie algebra is obtained.
At the same time no precise meaning
what "special limit" should mean for the general class of non-associative algebras satisfying the skew symmetry and the Hom-Jacobi identity (in our terminology) is given. When imposed on this specific example,
the meaning they use is exactly
that by passing to the limit $q\rightarrow 1$ in the defining commutation relations
one must recover the defining commutation relations of the classical Witt and Virasoro Lie algebras.
In 1999 in \cite{Hu}, the Example \ref{ex:qWitt} (up to minor change of deformation parameter and generators) of $q$-Witt and $q$-Virasoro algebras associated to Jackson $q$-derivative has been presented in a way close to the way it appeared in \cite{HLS}, namely using Jackson $q$-derivative by stressing the fact that it is a $\sigma$-derivation (skew-derivation), presenting the bracket multiplication and
the same $q$-Jacobi identity as described in the Example \ref{ex:qWitt} and then systematically using it
for obtaining the relations for $q$-Virasoro algebra as a central extension in terms of this bracket
and $q$-Jacobi identity. The article \cite{Hu} is mainly concerned with this specific example and some modifications and with the specific for this example non-associative structure, thus fitting perfectly
within the framework and the general method developed in \cite{HLS}.
The author of \cite{Hu} also used the specific $q$-deformed Jacobi condition \eqref{qWittJacobi}
satisfied in this particular example, the skew-symmetry, and the extra condition for the non-associative algebra being graded (but then not necessarily with one dimensional homogeneous components)
to define what he called $q$-Lie algebras. He also provided more examples of such $q$-Lie algebras
obtained as a sum of the $q$-Witt algebra and of quantum spaces.
The $q$-Jacobi condition \eqref{qWittJacobi}
is of cause a very specific example of the general operator twisted Jacobi condition of Hom-Lie algebras, and the condition that $q$-Lie algebras are graded is an additional restriction.
Therefore, the class of $q$-Lie algebras from \cite{Hu} is obviously just a specific
subclass of Hom-Lie algebras. Also clearly, in the limit when $q \rightarrow 1$ this particular $q$-deformed Jacobi condition converges to
Lie algebra Jacobi condition which means that the extra "special limit" condition in the 1990-1991 paper of Aizawa and Sato is obviously satisfied for $q$-Lie algebras in this sense, thus making them a subclass with respect to  the definition of Aizawa and Sato as well.
\end{remark}

\subsection{Hom-coalgebra, Hom-bialgebra and Hom-Hopf algebra structures}
Now we introduce the notions  of Hom-coalgebras,
Hom-bialgebras and Hom-Hopf algebras and describe
some properties of those structures extending the
classical structures of coalgebras, bialgebra
and Hopf algebra. We also define notions of
modules and comodules over  Hom-associative
algebra and Hom-coassociative
coalgebra.  Also, by removing the unit and the counit from the bialgebra structure, we obtain the notion of generalized bialgebra which we extend to our context. We prove, in the spirit of the theorem \ref{thmYauConstrHomAss}, that we can construct a Hom-associative coalgebra and a generalized Hom-bialgebra from a coalgebra (resp. generalized bialgebra) together with coalgebra morphism (resp. generalized bialgebra morphism).
\begin{definition}
A \emph{Hom-coalgebra} is a
triple $\left( V,\Delta ,\beta
\right) $ where $V$ is a $\K$-vector space and
$\Delta :V\rightarrow V\otimes V,\   \beta :V\rightarrow V
 $
are linear maps.

A  \emph{Hom-coassociative coalgebra} is a Hom-coalgebra satisfying
\begin{equation}\label{C1}  \left( \beta\otimes \Delta
\right) \circ \Delta =\left( \Delta \otimes
\beta\right) \circ \Delta.
\end{equation}
A Hom-coassociative coalgebra is said to be  \emph{counital} if there exists a map $\varepsilon :V\rightarrow \K $
satisfying

\begin{equation}\label{C2} \left( id\otimes \varepsilon
\right) \circ \Delta =id \qquad \text{ and }\qquad
\left( \varepsilon \otimes id\right) \circ
\Delta =id\end{equation}

\end{definition}

The conditions \ref{C1}  and \ref{C2} are respectively equivalent to the
following commutative diagrams:
\begin{eqnarray*}
\begin{array}{ccc}
 V & \stackrel{\Delta}{\longrightarrow
} & V\otimes
V \\
\ \quad \left\downarrow ^{ \Delta}\right. &  &
\quad \left\downarrow
^{\beta\otimes\Delta}  \right. \\
V\otimes V & \stackrel{\Delta \otimes
\beta}{\longrightarrow } & V\otimes V\otimes V
\end{array}
\quad\quad\quad
\begin{array}{ccccc} \K \otimes V &
\stackrel{\varepsilon \otimes
id_V}{\longleftarrow } & V\otimes V &
\stackrel{id\otimes \varepsilon }{\longrightarrow } & V\otimes \K \\
& \nwarrow ^{\cong } & \quad \left\uparrow
^\Delta \right. & \nearrow^{\cong
} &  \\
&  & V &  &
\end{array}
\end{eqnarray*}

Let $\left( V,\Delta ,\beta \right) $
and $\left( V^{\prime },\Delta ^{\prime
},\beta^{\prime }\right)
$ be two Hom-coalgebras (resp.  Hom-associative coalgebras). A
linear map  $f\ :V\rightarrow V^{\prime }$ is a \emph{morphism of  Hom-coalgebras} (resp. \emph{Hom-coassociative coalgebras}) if%
$$
(f\otimes f)\circ \Delta =\Delta ^{\prime }\circ
f
 \quad \text{ and}\quad f\circ\beta= \beta^{\prime }\circ f.$$

If furthermore the Hom-coassociative coalgebras admit counits $ \varepsilon$ and $ \varepsilon ^{\prime }$, we have moreover $\varepsilon =\varepsilon ^{\prime }\circ f $.

The following theorem shows how to construct a Hom-coassociative Hom-coalgebra starting from a coalgebra and a coalgebra endomorphism. It is a coalgebra version of the theorem \ref{thmYauConstrHomAss}. We  need only the coassociative comultiplication of the coalgebra. The existence of counit is not necessary.

\begin{theorem} \label{thmConstrHomCoalg}
Let $(V,\Delta)$ be a coalgebra  and let $\beta : V\rightarrow V$ be a coalgebra endomorphism. Then $(V,\Delta_\beta,\beta)$, where $\Delta_\beta=\Delta\circ\beta$,  is a Hom-coassociative coalgebra.

Moreover, suppose that  $(V',\Delta')$ is another coalgebra and  $\beta ' : V'\rightarrow V'$ is a coalgebra endomorphism. If $f:V\rightarrow V'$ is a coalgebra morphism that satisfies $f\circ\beta=\beta'\circ f$ then
$$f:(V,\Delta_\beta,\beta)\longrightarrow (V',\Delta'_{\beta'},\beta')
$$
is a morphism of Hom-coassociative coalgebras.
\end{theorem}
\begin{proof}
We show that $(V,\Delta_\beta,\beta)$ satisfies the axiom \ref{C1}.

 Indeed, using the fact that $(\beta\otimes\beta)\circ\Delta=\Delta\circ\beta$, we have
\begin{align*}
\left( \beta\otimes \Delta_\beta
\right) \circ \Delta_\beta &=\left( \Delta_\beta \otimes
\beta\right) \circ \Delta_\beta\\
&=\left( (\Delta\circ\beta) \otimes
\beta\right) \circ\Delta\circ\beta\\
&=\left(( (\beta\otimes\beta)\circ\Delta) \otimes
\beta\right) \circ\Delta\circ\beta\\
&= (\beta\otimes\beta\otimes\beta)\circ (\Delta \otimes
id ) \circ\Delta\circ\beta\\
&= (\beta\otimes\beta\otimes\beta)\circ (id \otimes\Delta
 ) \circ\Delta\circ\beta\\
 &=\left(\Delta_\beta \otimes\beta
\right) \circ \Delta_\beta.
\end{align*}

The second assertion is proved similarly:
$$f\circ\Delta_\beta=f\circ\Delta\circ\beta=\Delta'\circ f\circ\beta=\Delta'\circ\beta'\circ f=\Delta'_{\beta'}\circ f.$$
\end{proof}

We introduce  in the following,  the structures of
module and comodule over Hom-associative
algebras and Hom-coassociative
coalgebras and Hom-bialgebra.

Let $\A=(V,\mu,\alpha)$ be a Hom-associative
$\K$-algebra, an $\A$-module (left) is a triple
$(M,f,\gamma)$ where $M$ is $\K$-vector space and
$f,\gamma$ are  $\K$-linear maps, $f:  M
\rightarrow M$ and $\gamma : V \otimes M
\rightarrow M$, such that the following diagram
commutes:

$$
\begin{array}{ccc}
V\otimes V\otimes M & \stackrel{\mu \otimes
f}{\longrightarrow } & V\otimes
M \\
\ \quad \left\downarrow ^{\alpha \otimes \gamma
}\right. &  & \quad \left\downarrow
^\gamma \right. \\
V\otimes M & \stackrel{\gamma }{\longrightarrow }
& M
\end{array}
$$
The dualization leads to  comodule definition
over a Hom-coassociative coalgebra.

Let  $C=(V,\Delta, \beta)$ be a Hom-coassociative
coalgebra. A $C$-comodule (right) is a triple
$(M,g,\rho)$ where $M$ is a $\K$-vector space and
$g,\rho$ are $\K$-linear maps, $g: M \rightarrow
M$ and $\rho : M \rightarrow M\otimes V $, such
that the following diagram commutes:

$$
\begin{array}{ccc}
 M & \stackrel{\rho}{\longrightarrow } & M \otimes V
 \\
\ \quad \left\downarrow ^{\rho}\right. &  & \quad
\left\downarrow
^{g \otimes \Delta} \right. \\
M\otimes V & \stackrel{\rho\otimes
\beta}{\longrightarrow } & M\otimes V\otimes V
\end{array}
$$
\begin{remark}
A Hom-associative $\K$-algebra
$\A=(V,\mu,\alpha)$ is  a left $\A$-module with
$M=V$, $f=\alpha$ and $\gamma =\mu$. Also, a
Hom-coassociative coalgebra $C=(V,\Delta, \beta)$
is a right  $C$-comodule with $M=V$, $g=\beta$
and $\rho =\Delta$.
\end{remark}

\begin{definition}
A \emph{Hom-bialgebra} is a 7-uple $\left(
V,\mu ,\alpha,\eta ,\Delta ,\beta,\varepsilon
\right) $ where

(B1)\qquad $\left( V,\mu ,\alpha,\eta \right) $
is a Hom-associative algebra with unit  $\eta $.

(B2) $\qquad \left( V,\Delta ,\beta,\varepsilon
\right) $ is a Hom-coassociative coalgebra with counit $\varepsilon$.

(B3)\qquad The linear maps $\Delta $ and
$\varepsilon $ are compatible with the multiplication  $\mu$, that is $$
\left\{
\begin{array}{l}
\Delta \left( e_1\right) =e_1\otimes e_1\qquad
\text{where}\ e_1=\eta \left( 1\right) \\ \Delta
\left( \mu(x\otimes y)\right)=\Delta \left( x
\right) \bullet \Delta \left( y\right)
=\sum_{(x)(y)}\mu(x^{(1)}\otimes y^{\left(
1\right) })\otimes \mu( x^{\left( 2\right)
}\otimes y^{\left( 2\right) })\qquad \\
\varepsilon \left( e_1\right) =1 \\
\varepsilon \left( \mu(x\otimes y)\right)
=\varepsilon \left( x\right) \varepsilon \left(
y\right)
\end{array}
\right.
$$
where the bullet $\bullet$ denotes the
multiplication on tensor product and by  using
the Sweedler's notation $\Delta \left( x\right)
=\sum_{(x)}x^{(1)}\otimes x^{\left( 2\right) }$.
If there is no ambiguity we denote the
multiplication by a dot.
\end{definition}

\begin{remark}
One can consider a more restrictive definition
where  linear maps $\Delta $ and $\varepsilon $
are morphisms of Hom-associative algebras that is
the condition (B3) becomes equivalent to
$$
\left\{
\begin{array}{l}
\Delta \left( e_1\right) =e_1\otimes e_1\qquad
\text{where}\ e_1=\eta \left( 1\right) \\ \Delta
\left( \mu(x\otimes y)\right)=\Delta \left( x
\right) \bullet \Delta \left( y\right)
=\sum_{(x)(y)}\mu(x^{(1)}\otimes y^{\left(
1\right) })\otimes \mu( x^{\left( 2\right)
}\otimes y^{\left( 2\right) })\qquad \\
\varepsilon \left( e_1\right) =1 \\
\varepsilon \left( \mu(x\otimes y)\right)
=\varepsilon \left(
x\right) \varepsilon \left( y\right)\\
\Delta \left(
\alpha(x)\right)=\sum_{(x)}{\alpha(x^{(1)})\otimes
\alpha( x^{\left( 2\right) })}\\
\varepsilon\circ \alpha \left( x\right)
=\varepsilon \left( x\right)
\end{array}
\right.
$$
\end{remark}

 Given a Hom-bialgebra $\left( V,\mu ,\alpha,\eta ,\Delta,\beta ,\varepsilon
\right) $, we show that the vector space $Hom
\left( V,V \right)$ with the multiplication given
by the convolution product carries a structure of
Hom-associ\-ative algebra.
\begin{proposition}
Let  $\left( V,\mu ,\alpha,\eta ,\Delta,\beta
,\varepsilon \right) $ be a Hom-bialgebra. Then
the algebra $Hom \left( V,V \right)$ with the
multiplication given by the convolution product
defined by
$$ f \star g=\mu \circ \left( f\otimes g \right) \circ\Delta $$
and the unit being $\eta \circ \epsilon$ is a unital
Hom-associative algebra with the homomorphism map
defined by $\gamma (f)=\alpha \circ f \circ
\beta$.
\end{proposition}
\begin{proof}
Let $f,g,h\in Hom \left( V,V \right)$.
\begin{eqnarray*}
  \gamma (f) *( g*h)) &=& \mu \circ \left( \gamma (f)\otimes ( g*h) \right) \Delta  \\
   &=& \mu \circ \left( \gamma (f)\otimes ( \mu \circ \left( g\otimes h \right)\circ \Delta) \right) \Delta \\
   &=& \mu \circ \left( \alpha \otimes
   \mu \right)\circ
   \left( f  \otimes g\otimes h \right)
   \circ \left( \beta \otimes \Delta) \right)
   \Delta .
\end{eqnarray*}
Similarly
$$(f * g)*\gamma (h)=\mu \circ
\left(  \mu \otimes \alpha \right)\circ \left( f
\otimes g\otimes h \right)\circ \left( \Delta
\otimes \beta) \right) \Delta .
$$
Then, the Hom-associativity of $\mu$ and a
Hom-coassociativity of $\Delta$ lead to the
Hom-associ\-ativity of the convolution product.

The map $\eta \circ \epsilon$ is the unit for the convolution product. Indeed,  let $f\in Hom \left( V,V \right)$ and $x\in V$,
\begin{align*} (f \star (\eta \circ \epsilon))(x) =\mu \circ \left( f\otimes \eta \circ \epsilon \right) \circ\Delta(x)= \sum_{(x)}{\mu \left( f(x^{(1)})\otimes \eta \circ \epsilon(x^{(2)}) \right)}=\\ \sum_{(x)}{\epsilon(x^{(2)})\mu  \left( f(x^{(1)})\otimes \eta (1) \right)}=\sum_{(x)}{\epsilon(x^{(2)})  f(x^{(1)}) }=\sum_{(x)}{f( x^{(1)}\epsilon(x^{(2)})) }=f(x).
\end{align*}
Similar calculation shows that $ (\eta \circ \epsilon)\star f =f$.
\end{proof}
\begin{definition}
 An endomorphism $S$ of $V$ is said to be
an\emph{ antipode} if it is  the inverse of the
identity over $V$ for the Hom-associative algebra $Hom \left(
V,V \right)$ with the multiplication given by the
convolution product defined by
$$ f \star g=\mu \circ
\left( f\otimes g \right) \Delta $$ and the unit
being $\eta \circ \epsilon$.
\end{definition}

The condition being antipode may be expressed by
the condition:
$$
\mu \circ S\otimes Id\circ \Delta = \mu \circ
Id\otimes S\circ \Delta =\eta \circ \varepsilon .
$$
\begin{definition}
A {\it Hom-Hopf algebra }is a Hom-bialgebra with
an antipode.
\end{definition}
Then, a Hom-Hopf algebra over a $\K$-vector space
$V$ is given by
$$\HH=(V,\mu ,\alpha,\eta ,\Delta, \beta
,\varepsilon ,S)$$

where the following homomorphisms
$$\mu :V\otimes V\rightarrow V, \quad \eta :\K\rightarrow V,\quad \alpha :V\rightarrow V$$
$$\Delta :V\rightarrow V\otimes V, \quad \varepsilon :V\rightarrow \K,\quad\beta :V\rightarrow V$$
$$S :V\rightarrow \K$$
satisfy the following conditions
\begin{enumerate}
\item $(V,\mu,\alpha ,\eta )$ is a unital  Hom-associative algebra.
\item $(V,\Delta, \beta ,\varepsilon)$ is a counital Hom-coalgebra.
\item $\Delta $ and $\varepsilon $ are compatible with the multiplication  $\mu$, that is
$$
\left\{
\begin{array}{l}
\Delta \left( e_1\right) =e_1\otimes e_1\qquad
\text{where}\ e_1=\eta \left( 1\right) \\ \Delta
\left( x\cdot y\right)=\Delta \left( x \right)
\bullet \Delta \left( y\right)
=\sum_{(x)(y)}x^{(1)}\cdot y^{\left( 1\right)
}\otimes x^{\left( 2\right)
}\cdot y^{\left( 2\right) }\qquad \\
\varepsilon \left( e_1\right) =1 \\
\varepsilon \left( x\cdot y\right) =\varepsilon
\left( x\right) \varepsilon \left( y\right)
\end{array}
\right.
$$
\item $S$ is the antipode:
$$\mu \circ S\otimes Id\circ \Delta =\mu \circ Id\otimes S\circ \Delta =\eta
\circ \varepsilon .
$$
\end{enumerate}

Now we consider some antipode's properties of Hom-Hopf algebras.

Let $\HH=\left( V,\mu,\alpha ,\eta ,\Delta ,\beta,\varepsilon, S
\right) $  be a Hom-Hopf algebra. For any element $x\in V$, using the counity and
Sweedler notation, one may write
\begin{equation} x=\sum_{(x)}{x^{(1)}\otimes
\varepsilon (x^{(2)})}= \sum_{(x)}{\varepsilon
(x^{(1)})\otimes x^{(2)}}. \label{Coun}
\end{equation}
Then, for any $f\in End_\K (V)$, we have
\begin{equation}
f(x)=\sum_{(x)}{f(x^{(1)}) \varepsilon
(x^{(2)})}= \sum_{(x)}{\varepsilon
(x^{(1)})\otimes f(x^{(2)})}. \label{Coun2}
\end{equation}
Let $ f \star g=\mu \circ \left( f\otimes g
\right) \Delta $ be the convolution product of
$f,g\in End_\K(V)$. One may write
\begin{equation}
(f\star g)(x)= \sum_{(x)}{\mu ( f(x^{(1)})\otimes
g (x^{(2)}))}. \label{starpro}
\end{equation}
Since the antipode $S$ is the inverse of the
identity for the convolution product then $S$
satisfies
\begin{equation}
\varepsilon(x)\eta (1)= \sum_{(x)}{\mu (
S(x^{(1)})\otimes x^{(2)})}=\sum_{(x)}{\mu (
x^{(1)}\otimes S(x^{(2)})}). \label{anti1}
\end{equation}

\begin{proposition}
We have the following properties of the antipode when it exists :
\begin{enumerate}
\item The antipode $S$   is unique,
\item  $S(\eta (1))=\eta (1) $,
\item  $\varepsilon \circ S=\varepsilon$.
\end{enumerate}
\end{proposition}
\begin{proof}
1) We have $S\star id=id \star S=\eta \circ
\varepsilon$. Thus, $(S \star id) \star S=S \star
(id \star S)=S$. If $S'$ is another antipode of
$\HH$ then
$$S'=S' \star id \star S'=S' \star id
\star S=S \star id \star S=S.$$ Therefore the
antipode when it exists is unique.

2) Setting $e_1=\eta (1)$ and since $\Delta
(e_1)= e_1 \otimes e_1$ one has
$$(S\star id) (e_1)=
\mu (S(e_1)\otimes e_1) =S(e_1)=\eta(\varepsilon
(e_1))=e_1.$$

3) Applying \eqref{Coun2} to $S$, we obtain $
S(x)=\sum_{(x)}{S(x^{(1)}) \varepsilon
(x^{(2)})}$.

Applying $\varepsilon$ to \eqref{anti1}, we
obtain
$$\varepsilon (x)=\varepsilon
(\sum_{(x)}\mu(S(x^{(1)})\otimes x^{(2)})).$$
Since $\varepsilon$ is  compatible with the multiplication  $\mu$,
one has
$$\varepsilon (x)=\sum_{(x)}\varepsilon(S(x^{(1)}))\varepsilon
(x^{(2)})=\varepsilon(\sum_{(x)}S(x^{(1)})
\varepsilon(x^{(2)}))=\varepsilon(S(x)).
$$
Thus  $\varepsilon \circ S=\varepsilon$.
\end{proof}


\paragraph{Group-like and primitive elements.} Next, we describe some properties of group-like and
primitive elements in a Hom-bialgebra. We also define generalized primitive elements.
It turns out that the sets of primitive and  generalized primitive elements carry a natural structure of Hom-Lie algebra.

Let  $\HH=(V,\mu,\alpha ,\eta ,\Delta, \beta
,\varepsilon)$ be a Hom-bialgebra and $e_1=\eta
(1)$ be the unit.

\begin{definition}
An element $x\in \HH$ is called \emph{group-like} if $\Delta (x)= x\otimes x $ and  \emph{primitive} if
$\Delta (x)=e_1\otimes x + x\otimes e_1$.
\end{definition}

The coassociativity  of $\Delta$ on a group-like element $x\in \HH$ leads to the condition
$$\beta (x)\otimes x \otimes x= x\otimes x \otimes \beta (x)$$
The condition is satisfied, in particular, if $\beta (x)=\lambda x$ with $\lambda\in \K$.
Let  $V$ be an $n$-dimensional vector space, assume that $\{x_i\}_{i=1,\cdots,n}$ be a basis of $V$, the group-like element $x=\sum_{i=1}^{n}{a_i x_i}$ and
$\beta (x)=\sum_{i=1}^{n}{\beta_i x_i}$, then the coassociativity condition is equivalent to the system of equations
$$a_i \beta_j -a_j \beta_i= 0, \quad \ \text{for } \ \  i,j=1,\cdots,n.$$
For the unit $e_1$, which is also a group-like element, the coassociativity implies  $$\beta (e_1)=\lambda e_1 \quad \lambda\in \K.$$

Now, let $x\in \HH$ be a primitive element, the
coassociativity of $\Delta$ on $x$  implies
$$e_1\otimes e_1\otimes (\beta(x)-\lambda x)+(\beta(x)-\lambda x)\otimes e_1\otimes e_1=0.
$$
Therefore, for a primitive element $x\in \HH$, we have $$\beta(x)=\lambda x.$$
Also, for a primitive element $x\in \HH$, one has
$$(\beta\otimes\Delta)\circ\Delta(x)=\tau_{13}\circ(\Delta\otimes\beta)\circ\Delta(x)
$$
where $\tau_{13}$ is a permutation in the
symmetric group $\mathcal{S}_3$.
Thus, the coassociativity condition becomes
$$(\Delta\otimes\beta)\circ\Delta(x)=\tau_{13}\circ(\Delta\otimes\beta)\circ\Delta(x)
$$
\begin{lemma}
Let $x$ be a primitive element in $\HH$, then
$\varepsilon(x)=0$.
\end{lemma}
\begin{proof}
By counity property, we have $ x=(id \otimes
\varepsilon)\circ \Delta (x)$. If $\Delta
(x)=e_1\otimes x + x\otimes e_1$, then
$x=\varepsilon (x) e_1+\varepsilon (e_1) x$, and
since $\varepsilon (e_1)=1$ it implies
$\varepsilon (x)=0$.
\end{proof}
\begin{proposition}

Let  $\HH=(V,\mu,\alpha ,\eta ,\Delta, \beta
,\varepsilon)$ be a Hom-bialgebra and $e_1=\eta
(1)$ be the unit.
 If $x$ and $y$ are two primitive elements in $\HH$. Then we have $\varepsilon (x)=0$ and  the commutator
 $[x,y]=\mu (x\otimes
y) -\mu (y \otimes x)$ is also a primitive
element.

The set of all primitive elements of $\HH$,
denoted by $Prim(\HH)$, has a structure of
Hom-Lie algebra.
\end{proposition}
\begin{proof}
 By a direct calculation one has
\begin{eqnarray*}
   \Delta ([x,y]) &=& \Delta ( \mu (x\otimes y) -\mu (y \otimes x))\\
 \ &=& \Delta (x)\bullet \Delta (y)-\Delta (y)\bullet \Delta (x) \\
  \ &=& (e_1\otimes
x + x\otimes e_1)\bullet (e_1\otimes y + y\otimes
e_1)- (e_1\otimes y + y\otimes e_1)\bullet
(e_1\otimes
x + x\otimes e_1)\\
 \ &=& e_1 \otimes \mu(x\otimes y)+y\otimes x+x\otimes y+\mu(x\otimes y)\otimes e_1\\ & &
 -e_1 \otimes \mu(y\otimes x )-x\otimes y-y\otimes x -\mu( y\otimes x)\otimes e_1\\
  \ &=& e_1 \otimes (\mu(x\otimes y)- \mu(y\otimes x ))+(\mu(x\otimes y)- \mu(y\otimes x ))\otimes e_1\\
  \ &=& e_1 \otimes
[x,y]+[x,y]\otimes e_1
\end{eqnarray*}

which means that $Prim(\HH)$ is closed under the
bracket multiplication $[\cdot,\cdot]$.

We have seen in \cite{MS} that there is a natural
map from the Hom-associative algebras to Hom-Lie
algebras. The bracket $[x,y]=\mu (x\otimes y)
-\mu (y \otimes x)$ is obviously skewsymmetric
and one checks that the Hom-Jacobi condition is
satisfied:

\begin{multline}
  [\alpha(x),[y,z]]-[[x,y],\alpha(z)]-[\alpha (y),[x,z]]=  \nonumber \\
   \mu (\alpha(x)\otimes \mu (y\otimes z))-\mu (\alpha(x)\otimes \mu (z\otimes y
   ))
   -\mu (\mu (y\otimes z)\otimes \alpha(x))+\mu (\mu (z\otimes y )\otimes \alpha(x))\nonumber \\
   -\mu (\mu (x\otimes y)\otimes \alpha(z)) +\mu (\mu (y\otimes x )\otimes \alpha(z))+
   \mu (\alpha(z)\otimes \mu (x\otimes y))-\mu (\alpha(z)\otimes \mu (y\otimes
   x
   ))\nonumber \\
   -\mu (\alpha(y)\otimes \mu (x\otimes z))+\mu (\alpha(y)\otimes \mu (z\otimes
   x
   ))
   +\mu (\mu (x\otimes z)\otimes \alpha(y))-\mu (\mu (z\otimes x )\otimes
   \alpha(y))=0
\end{multline}

\end{proof}

We introduce now a notion of generalized
primitive element.
\begin{definition}
An element $x\in \HH$ is called generalized
primitive element if it satisfies the conditions
\begin{equation}
\label{GPrim}(\beta\otimes\Delta)\circ\Delta(x)
=\tau_{13}\circ(\Delta\otimes\beta)\circ\Delta(x)
\end{equation}
\begin{equation}
\Delta^{op}(x)=\Delta(x)
\end{equation}

where $\tau_{13}$ is a permutation in the
symmetric group $\mathcal{S}_3$.
\end{definition}

In particular, a primitive element in $\HH$ is a generalized
primitive element.
\begin{remark}
The condition \eqref{GPrim} may be written
$$(\Delta\otimes\beta)\circ\Delta(x)=
\tau_{13}\circ(\beta\otimes\Delta)\circ\Delta(x).
$$
\end{remark}
\begin{proposition}
Let  $\HH=(V,\mu,\alpha ,\eta ,\Delta, \beta
,\varepsilon)$ be a Hom-bialgebra and $e_1=\eta
(1)$ be the unit.
 If $x$ and $y$ are two generalized primitive elements in $\HH$.
 Then, we have $\varepsilon (x)=0$ and the commutator $[x,y]=\mu (x\otimes y) -\mu (y
\otimes x)$ is also a generalized primitive
element.

The set of all generalized primitive elements of
$\HH$, denoted by $GPrim(\HH)$, has a structure
of Hom-Lie algebra.
\end{proposition}
\begin{proof}
Let $x$ and $y$ be two generalized primitive
elements in $\HH$. In the following the
multiplication $\mu$ is denoted by a dot. The
following equalities hold:
\begin{eqnarray*}
  (\Delta\otimes\beta)\circ\Delta(x\cdot y-y\cdot x) &=& (\Delta\otimes\beta)\circ\Delta(x\cdot y)-
  (\Delta\otimes\beta)\circ\Delta(y\cdot x) \\
  \ &=& (\Delta\otimes\beta)(\Delta(x)\bullet \Delta(y))- (\Delta\otimes\beta)(\Delta(y)\bullet \Delta(x))\\
 \ &=& \Delta(x^{(1)}\cdot y^{(1)})\otimes \beta(x^{(2)}\cdot y^{(2)})-
 \Delta(y^{(1)}\cdot x^{(1)})\otimes \beta(y^{(2)}\cdot x^{(2)})\\
  \ &=& (x^{(1)(1)}\cdot y^{(1)(1)})\otimes(x^{(1)(2)}\cdot y^{(1)(2)})\otimes \beta(x^{(2)}\cdot
  y^{(2)})\\ && -(y^{(1)(1)}\cdot x^{(1)(1)})\otimes(y^{(1)(2)}\cdot x^{(1)(2)})\otimes \beta(y^{(2)}\cdot
  x^{(2)}).
\end{eqnarray*}
Then, using the fact that $\Delta^{op}=\Delta$
for generalized primitive elements one has:
\begin{eqnarray*}
  \tau_{13}\circ(\Delta\otimes\beta)\circ\Delta(x\cdot y-y\cdot x) &=&\beta(x^{(2)}\cdot
  y^{(2)})\otimes(x^{(1)(2)}\cdot y^{(1)(2)})\otimes
  (x^{(1)(1)}\cdot y^{(1)(1)}) \\
  &&-\beta(y^{(2)}\cdot
  x^{(2)})\otimes(y^{(1)(2)}\cdot x^{(1)(2)})\otimes (y^{(1)(1)}\cdot x^{(1)(1)})  \\
  \ &=& (\beta\otimes\Delta)
  \circ\Delta(x\cdot y-y\cdot x).
\end{eqnarray*}
The structure of Hom-Lie algebra follows from the
same argument as in the primitive elements case.
\end{proof}

\paragraph{Generalized Hom-bialgebras.} We introduce in the following a more general definition of Hom-bialgebra, where we do not consider that the Hom-associative algebra is unital and the Hom-coassociative coalgebra is counital. This notion extend the notion of generalized bialgebra in the associative case introduced by Loday \cite{Loday} in a more general framework. A generalized bialgebra is a triple  $\left(V,\mu,\Delta \right) $ of an associative multiplication $\mu$, a coassociative comultiplication $\Delta$ together with a compatibility condition.
\begin{definition}
A \emph{generalized Hom-bialgebra} is a 5-uple $\left(
V,\mu ,\alpha ,\Delta ,\beta
\right) $ where

(GB1)\qquad $\left( V,\mu ,\alpha \right) $
is a Hom-associative algebra.

(GB2) $\qquad \left( V,\Delta ,\beta
\right) $ is a Hom-coassociative coalgebra.

(GB3)\qquad The linear maps $\Delta $ is compatible with the multiplication  $\mu$, that is
$$ \Delta
\left( \mu(x\otimes y)\right)=\Delta \left( x
\right) \bullet \Delta \left( y\right)
=\sum_{(x)(y)}\mu(x^{(1)}\otimes y^{\left(
1\right) })\otimes \mu( x^{\left( 2\right)
}\otimes y^{\left( 2\right) })
$$
\end{definition}
We call generalized bialgebra a generalized Hom-bialgebra where $\alpha$ and $\beta$ are the identity map. Any bialgebra is in particular a generalized bialgebra.  A morphism of generalized Hom-bialgebras is just a morphism of Hom-associative algebras and a morphism of Hom-coassociative coalgebras.  We can construct generalized Hom-bialgebras using generalized bialgebras or just bialgebras with bialgebra endomorphisms.
\begin{proposition}\label{propoGenHomBialg}
Let $(V,\mu,\Delta)$ be a generalized bialgebra  and let $\alpha : V\rightarrow V$ be a generalized bialgebra endomorphism. Then $(V,\mu_\alpha,\alpha,\Delta_\alpha,\alpha)$, where $\mu_\alpha=\alpha\circ\mu$ and $\Delta_\alpha=\Delta\circ\alpha$,  is a generalized Hom-bialgebra.

Moreover, suppose that  $(V',\mu',\Delta')$ is another generalized bialgebra and  $\alpha ' : V'\rightarrow V'$ is a generalized bialgebra endomorphism. If $f:V\rightarrow V'$ is a generalized bialgebra morphism that satisfies $f\circ\alpha=\alpha'\circ f$ then
$$f:(V,\mu_\alpha,\alpha,\Delta_\alpha,\alpha)\longrightarrow (V',\mu'_{\alpha'},\alpha',\Delta_{\alpha'},\alpha')
$$
is a morphism of generalized Hom-bialgebras.

\end{proposition}
\begin{proof}
The condition GB1  follows from the theorem \ref{thmYauConstrHomAss} and GB2 from the theorem \ref{thmConstrHomCoalg}.
It remains to prove the compatibility condition (GB3).  The  condition may be written
$$ \Delta_\alpha
\circ \mu_\alpha=(\mu_\alpha\otimes\mu_\alpha)\circ \Upsilon \circ (\Delta_\alpha\otimes\Delta_\alpha)
$$
where $\Upsilon$ is the usual flip, $\Upsilon(x_1\otimes x_2\otimes x_3 \otimes x_4)=x_1\otimes x_3\otimes x_2 \otimes x_4$.
We have
\begin{align*}
 \Delta_\alpha\circ
 \mu_\alpha &= \Delta\circ\alpha\circ
 \alpha\circ\mu=(\alpha\otimes\alpha)\circ\Delta\circ\mu\circ
( \alpha\otimes  \alpha)\\
&=(\alpha\otimes\alpha)\circ((\mu\otimes\mu)\circ \Upsilon \circ (\Delta\otimes\Delta))\circ
( \alpha\otimes  \alpha)\\
&=((\alpha\circ\mu)\otimes(\alpha\circ\mu)\circ \Upsilon \circ( (\Delta\circ\alpha)\otimes\Delta\circ\alpha))\\
&=(\mu_\alpha\otimes\mu_\alpha)\circ \Upsilon \circ (\Delta_\alpha\otimes\Delta_\alpha)
\end{align*}
The second assertion follows also from theorems \ref{thmYauConstrHomAss} and  \ref{thmConstrHomCoalg}.
\end{proof}

Now, we provide some examples of generalized Hom-bialgebra.
\begin{example}
Let   $\K G$ be the group-algebra over the group $G$.
As a vector space, $\K G$ is generated by   $\{e_g : g \in
G\}$. If $\alpha:G\rightarrow G$ is a group homomorphism, then it can be extended to an algebra endomorphism of  $\K G$ by setting
$$\alpha (\sum_{g\in G}{a_g e_g})=\sum_{g\in G}{a_g \alpha (e_g)}.
$$
A structure of Hom-associative algebra over  $\K G$ was defined in \cite{Yau:HomolHomLie} using theorem \ref{thmYauConstrHomAss}. Consider the usual bialgebra structure on $\K G$ and $\alpha$  a generalized bialgebra morphism.  Then,  following the proposition \ref{propoGenHomBialg},  we define over $\K G$ a  generalized Hom-bialgebra $(\K G, \mu, \alpha,\Delta,\alpha)$  by setting:
$$\mu(e_g\otimes e_{g'})=\alpha (e_{g\cdot g'}),
$$
$$
\Delta \left( e_{g}\right) =\alpha(e_{g})\otimes \alpha(e_{g}).
$$
\end{example}

\begin{example} Consider the  polynomial algebra  $\A=\K [(X_{ij})]$ in variables $(X_{i j})_{i,j=1,\cdots,n}$. It carries a  structure of generalized bialgebra with the comultiplication  defined by
$ \delta (X_{i j})=\sum_{k=1}^{n}{X_{i k}\otimes X_{k j}}$ and  $\delta (1)=1\otimes 1$.
Let $\alpha$ be a generalized bialgebra morphism, it is defined by $n^2$ polynomials $\alpha (X_{i j})$.
We define a generalized Hom-bialgebra $(\A, \mu, \alpha,\Delta,\alpha)$ by
\begin{align*}
\mu(f\otimes g)&=f(\alpha (X_{1 1}),\cdots,\alpha (X_{n n}))g(\alpha (X_{1 1}),\cdots,\alpha (X_{n n})),\\
 \Delta (X_{i j})&=\sum_{k=1}^{n}{\alpha (X_{i k})\otimes \alpha (X_{k j})},\\
 \Delta (1)&=\alpha (1)\otimes \alpha (1).
\end{align*}
\end{example}

\begin{example}
Let $X$ be a set and consider the set of non-commutative
polynomials $\A=\K\langle X \rangle$. It carries  a generalized bialgebra structure with a
 comultiplication defined for  $ x\in X$
by  $\delta (x)=1\otimes x +x\otimes 1$ and  $\delta (1)=1\otimes 1$.
Let $\alpha$ be a generalized bialgebra morphism.
We define a generalized Hom-bialgebra $(\A, \mu, \alpha,\Delta,\alpha)$ by
\begin{align*}
\mu(f\otimes g)&=f(\alpha (X))g(\alpha (X)),\\
 \Delta (x)&=\alpha (1)\otimes \alpha (x) +\alpha (x)\otimes \alpha (1),\\
 \Delta (1)&=\alpha (1)\otimes \alpha (1).
\end{align*}

\end{example}

\begin{remark}The previous constructions show that Hom-associative algebra (resp.  the Hom-coassociative coalgebra and the generalized Hom-bialgebra) may be viewed as a deformation of the associative algebra (resp.  coalgebra and generalized bialgebra). We recover the first structure when the endomorphism $\alpha$ becomes the identity map.
\end{remark}

\section{Hom-Lie admissible Hom-Coalgebras}
\label{sec:HomLieadmisHomcoalg}
 Let $(V,\Delta, \beta )$ be a Hom-coalgebra
where $V$ is a vector space over
$\K$, $\Delta  : V\rightarrow V\otimes V$ and
$\beta : V\rightarrow V$ are linear maps and
$\Delta$  is not necessarily coassociative or
Hom-coassociative.

By a  \emph{$\beta$-coassociator} of $\Delta$ we
call a linear map $\textbf{c}_\beta (\Delta)$
defined by
$$\textbf{c}_\beta (\Delta):= \left(
\Delta \otimes \beta\right) \circ \Delta- \left(
\beta\otimes \Delta \right) \circ \Delta.
$$
Let $\mathcal{S}_3$ be the symmetric group of
order 3. Given $\sigma\in \mathcal{S}_3 $, we
define a linear map
$$\Phi _ \sigma \ : V^{\otimes 3}\longrightarrow V^{\otimes 3}
$$
by
$$\Phi _ \sigma (x_1\otimes x_2 \otimes x_3 )= x_{\sigma ^{-1}(1)}  \otimes x_{\sigma
^{-1}(2)}\otimes x_{\sigma ^{-1}(3)}.
$$
Recall that $\Delta^{op}=\tau\circ \Delta$ where
$\tau$ is the usual flip that is $\tau (x\otimes
y)= y \otimes x$.
\begin{definition} \label{def:HomLieadmissibleHomcoalg}
  A triple
$( V, \Delta, \beta)$   is a \emph{Hom-Lie
admissible Hom-coalgebra} if the linear map
$$\Delta_L : V\longrightarrow V\otimes V
$$
defined by $\Delta _L = \Delta - \Delta^{op} $,
is a Hom-Lie coalgebra multiplication, that is
the following condition is satisfied
\begin{eqnarray}\label{CoHomLie}
  \textbf{c}_\beta (\Delta_L)+\Phi_{(213)} \circ \textbf{c}_\beta (\Delta_L)+
   \Phi_{(231)} \circ \textbf{c}_\beta (\Delta_L) =
  0
\end{eqnarray}
where $(213)$ and $(231)$ are the two cyclic
permutations of order 3 in $\mathcal{S}_3$.
\end{definition}
\begin{remark}
 Since $\Delta _L = \Delta - \Delta^{op} $,
 the equality $ \Delta_L ^{op}= - \Delta_L$ holds.
 \end{remark}

\begin{lemma}
Let $(V,\Delta, \beta )$ be a Hom-coalgebra where
$\Delta  : V\rightarrow V\otimes V$ and $\beta :
V\rightarrow V$ are linear maps and $\Delta$  is
not necessarily coassociative or
Hom-coassociative, then the following relations
are true
\begin{eqnarray}
\textbf{c}_\beta (\Delta^{op})&=&-\Phi_{(13)}
\circ \textbf{c}_\beta
(\Delta) \\
(\beta \otimes \Delta^{op})\circ \Delta &=&
\Phi_{(13)} \circ (\Delta \otimes \beta) \circ
\Delta^{op}
 \\
 (\beta \otimes \Delta)\circ \Delta ^{op} &=& \Phi_{(13)} \circ
(\Delta ^{op} \otimes \beta) \circ \Delta
 \\
 (\Delta \otimes \beta)\circ \Delta ^{op} &=& \Phi_{(213)} \circ
( \beta\otimes \Delta ) \circ \Delta
 \\
 (\Delta  ^{op}\otimes \beta)\circ \Delta &=& \Phi_{(12)} \circ
( \Delta\otimes \beta ) \circ \Delta .
\end{eqnarray}
\end{lemma}
\begin{lemma}
The $\beta$-coassociator of $\Delta_{L}$ is
expressed using $\Delta$ and $\Delta ^{op}$ as
follows:
\begin{eqnarray}
\label{coassDeltaLviaDeltaDeltaOp}
\textbf{c}_\beta (\Delta_{L})&=&\textbf{c}_\beta
(\Delta)+\textbf{c}_\beta (\Delta^{op})\\
\nonumber && -(\Delta \otimes \beta)\circ \Delta
^{op}-(\Delta  ^{op}\otimes \beta)\circ \Delta
+\\ \nonumber  && \Phi_{(13)} \circ (\Delta
\otimes \beta) \circ \Delta^{op}+\Phi_{(13)}
\circ ( \Delta ^{op} \otimes \beta )
\circ \Delta \\
\label{coassDeltaLviaDelta} &=&\textbf{c}_\beta
(\Delta)-\Phi_{(13)} \circ \textbf{c}_\beta (\Delta)\\
\nonumber && -\Phi_{(213)} \circ (\beta \otimes
\Delta)\circ \Delta -\Phi_{(12)} \circ (\Delta
\otimes \beta)\circ \Delta \\ \nonumber &&
+\Phi_{(23)} \circ (\beta \otimes \Delta) \circ
\Delta +\Phi_{(231)} \circ ( \Delta \otimes \beta
) \circ \Delta .
\end{eqnarray}
\end{lemma}

\begin{proposition} \label{prop:coJacobiSumA3viaS3}
Let  $( V, \Delta, \beta)$ be  a Hom-coalgebra.
Then one has
\begin{equation} \label{coJacobiSumA3viaS3}
\textbf{c}_\beta (\Delta_L)+\Phi_{(213)} \circ
\textbf{c}_\beta (\Delta_L)+
   \Phi_{(231)} \circ \textbf{c}_\beta (\Delta_L) = 2 \sum_{\sigma \in \mathcal{S}_3} {(-1)^{\epsilon (\sigma)}\Phi
_\sigma \circ \textbf{c}_\beta (\Delta)}
\end{equation}
where $(-1)^{\epsilon (\sigma)}$ is the signature
of the permutation $\sigma$.
\end{proposition}
\begin{proof}
By \eqref{coassDeltaLviaDelta}  and
multiplication rules in the group $\mathcal{S}_3$, it
follows that
\begin{eqnarray} \nonumber \Phi_{(213)} \circ
\textbf{c}_\beta (\Delta_L) &=& \Phi_{(213)}
\circ \textbf{c}_\beta (\Delta)
-\Phi_{(213)}\circ \Phi_{(13)} \circ \textbf{c}_\beta (\Delta)\\
\nonumber && -\Phi_{(213)}\circ\Phi_{(213)} \circ
(\beta \otimes \Delta)\circ \Delta
-\Phi_{(213)}\circ\Phi_{(12)} \circ (\Delta
\otimes \beta)\circ \Delta \\ \nonumber &&
+\Phi_{(213)}\circ\Phi_{(23)} \circ (\beta
\otimes \Delta) \circ \Delta
+\Phi_{(213)}\circ\Phi_{(231)} \circ ( \Delta
\otimes \beta )
\circ \Delta  \\
\label{Phi(213)coassDeltaLviaDelta} &=&
\Phi_{(213)} \circ \textbf{c}_\beta (\Delta)
-\Phi_{(12)}\circ \textbf{c}_\beta (\Delta)\\
\nonumber && -\Phi_{(231)} \circ (\beta \otimes
\Delta)\circ \Delta -\Phi_{(23)} \circ (\Delta
\otimes \beta)\circ \Delta \\ \nonumber &&
+\Phi_{(13)} \circ (\beta \otimes \Delta) \circ
\Delta + ( \Delta
\otimes \beta ) \circ \Delta, \\
\nonumber \Phi_{(231)} \circ \textbf{c}_\beta
(\Delta_L) &=& \Phi_{(231)} \circ
\textbf{c}_\beta (\Delta)
-\Phi_{(231)}\circ \Phi_{(13)} \circ \textbf{c}_\beta (\Delta)\\
\nonumber && -\Phi_{(231)}\circ\Phi_{(213)} \circ
(\beta \otimes \Delta)\circ \Delta
-\Phi_{(231)}\circ\Phi_{(12)} \circ (\Delta
\otimes \beta)\circ \Delta \\ \nonumber &&
+\Phi_{(231)}\circ\Phi_{(23)} \circ (\beta
\otimes \Delta) \circ \Delta
+\Phi_{(231)}\circ\Phi_{(231)} \circ ( \Delta
\otimes \beta )
\circ \Delta  \\
\label{Phi(231)coassDeltaLviaDelta} &=&
\Phi_{(231)} \circ \textbf{c}_\beta (\Delta)
-\Phi_{(23)}\circ \textbf{c}_\beta (\Delta)\\
\nonumber && -(\beta \otimes \Delta)\circ \Delta
-\Phi_{(13)} \circ (\Delta \otimes \beta)\circ
\Delta \\ \nonumber && +\Phi_{(12)} \circ (\beta
\otimes \Delta) \circ \Delta + \Phi_{(213)} \circ
(\Delta \otimes \beta ) \circ \Delta.
\end{eqnarray}
After summing up the equalities
\eqref{coassDeltaLviaDelta},
\eqref{Phi(213)coassDeltaLviaDelta} and
\eqref{Phi(231)coassDeltaLviaDelta} the terms on
the right hand sides may be pairwise combined
into the terms of the form ${(-1)^{\epsilon
(\sigma)}\Phi _\sigma \circ \textbf{c}_\beta
(\Delta)}$ with each one being present in the sum
twice for all $\sigma \in \mathcal{S}_3$.
\end{proof}

Definition \ref{def:HomLieadmissibleHomcoalg}
together with \eqref{coJacobiSumA3viaS3} yields
the following corollary.

\begin{corollary}
A triple $( V, \Delta, \beta)$ is a Hom-Lie
admissible Hom-coalgebra if and only if
$$\sum_{\sigma \in \mathcal{S}_3}{(-1)^{\epsilon (\sigma)}\Phi
_\sigma \circ \textbf{c}_\beta (\Delta)}=0
$$
where $(-1)^{\epsilon (\sigma)}$ is the signature
of the permutation $\sigma$.
\end{corollary}

Next we introduce the notion of
\emph{$G$-Hom-coalgebra} where $G$ is any subgroup
of the symmetric group $\mathcal{S}_3$.
\begin{definition}
Let $G$ be a subgroup of the symmetric group
$\mathcal{S}_3$, A Hom-coalgebra $( V, \Delta,
\beta)$ is called \emph{$G$-Hom-coalgebra} if
\begin{equation}\label{admi}
\sum_{\sigma \in G}{(-1)^{\epsilon (\sigma)}\Phi
_\sigma \circ \textbf{c}_\beta (\Delta)}=0
\end{equation}
where  $(-1)^{\varepsilon ({\sigma})}$ is the
signature of the permutation $\sigma$.
\end{definition}

\begin{proposition}
Let $G$ be a subgroup of the permutations group
$\mathcal{S}_3$. Then any $G$-Hom-Coalgebra $( V,
\Delta, \beta)$ is a Hom-Lie admissible
Hom-coalgebra.
\end{proposition}
\begin{proof}
The skew-symmetry follows straightaway from the
definition. Take the set of conjugacy classes
$\{g G\}_{g\in I}$  where $I\subseteq G$, and for
any $\sigma_1, \sigma_2\in I,\sigma_1 \neq
\sigma_2 \Rightarrow \sigma_1 G\bigcap \sigma_1 G
=\emptyset$. Then

$$\sum_{\sigma\in \mathcal{S}_3}{{(-1)^{\epsilon (\sigma)}}\Phi _\sigma \circ
\textbf{c}_\beta (\Delta)}=\sum_{\sigma_1\in
I}{\sum_{\sigma_2\in \sigma_1 G}{(-1)^{\epsilon
(\sigma)}\Phi _\sigma \circ \textbf{c}_\beta
(\Delta)}}=0.$$
\end{proof}

The subgroups of $\mathcal{S}_3$ are
$$G_1=\{Id\}, ~G_2=\{Id,\tau_{1 2}\},~G_3=\{Id,\tau_{2
3}\},$$ $$~G_4=\{Id,\tau_{1 3}\},~G_5=A_3
,~G_6=\mathcal{S}_3,$$ where $A_3$ is the
alternating group and where $\tau_{ij}$ is the
transposition between $i$ and $j$.

We obtain the following type of
Hom-Lie-admissible Hom-coalgebras.
\begin{itemize}
\item The  $G_1$-Hom-coalgebras  are the Hom-associative
coalgebras defined above.

\item The  $G_2$-Hom-coalgebras satisfy the condition
\begin{equation}\nonumber
\textbf{c}_\beta (\Delta)+
\Phi_{(12)}\textbf{c}_\beta (\Delta)=0.
\end{equation}

\item The  $G_3$-Hom-coalgebras satisfy the condition
\begin{equation}\nonumber
\textbf{c}_\beta (\Delta)+
\Phi_{(23)}\textbf{c}_\beta (\Delta)=0.
\end{equation}

\item The  $G_4$-Hom-coalgebras satisfy the condition
\begin{equation}\nonumber
\textbf{c}_\beta (\Delta)+
\Phi_{(13)}\textbf{c}_\beta (\Delta)=0.
\end{equation}
\item The  $G_5$-Hom-coalgebras satisfy the condition
\begin{eqnarray}\nonumber
\textbf{c}_\beta (\Delta)+
\Phi_{(213)}\textbf{c}_\beta (\Delta)+
\Phi_{(231)}\textbf{c}_\beta (\Delta)=0.
\end{eqnarray}\nonumber
If the product $\mu$ is skew-symmetric then the
previous condition is exactly the Hom-Jacobi
identity.
\item The  $G_6$-Hom-coalgebras are the Hom-Lie-admissible
coalgebras.
\end{itemize}

The $G_2$-Hom-coalgebras may be called
Vinberg-Hom-coalgebra and $G_3$-Hom-coalgebras
may be called preLie-Hom-coalgebras.

\begin{definition}
 A triple $(V,
\Delta, \beta)$ consisting of a linear space $V$,
a linear map $\mu: V \rightarrow V\times V$ and a
homomorphism $\beta$  is called a  \emph{Vinberg-Hom-coalgebra} if it satisfies
\begin{equation}\nonumber
\textbf{c}_\beta (\Delta)+
\Phi_{(12)}\textbf{c}_\beta (\Delta)=0.
\end{equation}
and is called  a  \emph{preLie-Hom-coalgebra} if it satisfies
\begin{equation}\nonumber
\textbf{c}_\beta (\Delta)+
\Phi_{(23)}\textbf{c}_\beta (\Delta)=0.
\end{equation}
\end{definition}

More generally, by dualization we have a
correspondence between $G$-Hom-associative
algebras introduced in \cite{MS} and
$G$-Hom-coalgebras for a subgroup $G$ of
$\mathcal{S}_3$.

Let $G$ be a subgroup  of $\mathcal{S}_3$, and let
$(V, \mu, \alpha)$ be a $G$-Hom-associative
algebra that is $\mu : V\otimes V \rightarrow V$
and $ \alpha : V \rightarrow V$ are linear maps
and the following condition is satisfied
\begin{equation}\label{G_ass}
\sum_{\sigma\in G}{(-1)^{\varepsilon
({\sigma})}a_{\alpha,\mu}\circ \Phi_\sigma}=0.
\end{equation}
where $a_{\alpha,\mu}$ is the $\alpha$-associator
that is $a_{\alpha,\mu}= \mu \circ(\mu  \otimes
\alpha) -\mu\circ(\alpha\otimes \mu) $

Setting $$(\mu \otimes \alpha)_G =\sum_{\sigma\in
G}{(-1)^{\varepsilon ({\sigma})}(\mu \otimes
\alpha)\circ \Phi_\sigma}\  \text{ and } \
(\alpha \otimes \mu )_G =\sum_{\sigma\in
G}{(-1)^{\varepsilon ({\sigma})}(\alpha \otimes
\mu )\circ \Phi_\sigma}$$ the condition
\eqref{G_ass} is equivalent to the following
commutative diagram
$$
\begin{array}{ccc}
V\otimes V\otimes V & \stackrel{(\mu \otimes
\alpha)_G}{\longrightarrow } & V\otimes
V \\
\ \quad \left\downarrow ^{(\alpha \otimes \mu
)_G}\right. &  & \quad \left\downarrow
^\mu \right. \\
V\otimes V & \stackrel{\mu }{\longrightarrow } &
V
\end{array}
$$

By the dualization of the square one may obtain
the following commutative diagram
$$
\begin{array}{ccc}
 V & \stackrel{\Delta}{\longrightarrow
} & V\otimes
V \\
\ \quad \left\downarrow ^{ \Delta}\right. &  &
\quad \left\downarrow
^{(\beta\otimes\Delta)_G}  \right. \\
V\otimes V & \stackrel{(\Delta \otimes
\beta)_G}{\longrightarrow } & V\otimes V\otimes V
\end{array}
$$
where $$ (\beta\otimes\Delta)_G=\sum_{\sigma\in
G}{(-1)^{\varepsilon ({\sigma})}\Phi_\sigma \circ
(\beta \otimes\Delta )}  \ \text{ and }\ (\Delta
\otimes \beta)_G=\sum_{\sigma\in
G}{(-1)^{\varepsilon ({\sigma})}\Phi_\sigma \circ
(\Delta\otimes\beta)}.
$$

The previous commutative diagram expresses that
$(V, \Delta, \beta)$ is a $G$-Hom-coalgebra. More
precisely we have the following connection
between $G$-Hom-coalgebras and
$G$-Hom-associ\-ative algebras.

\begin{proposition}
Let  $\left( V,\Delta ,\beta \right) $ be a
$G$-Hom-coalgebra where $G$ is a subgroup of
$\mathcal{S}_3$. Its dual vector space $V^\star$
is provided with a $G$-Hom-associ\-ative algebra
$\left( V^\star,\Delta^\star ,\beta^\star \right)
$ where $\Delta^\star ,\beta^\star$ are the
transpose map.
\end{proposition}
\begin{proof}
Let  $\left( V,\Delta , \beta \right) $ be a
$G$-Hom-coalgebra, and let $V^{\star}$ be the dual
space of $V$, that is $V^{\star}=Hom(V,\K)$.

Consider the map
$$\lambda_n : (V^{\star})^{\otimes n}\longrightarrow (V^{\star})^{\otimes n}
$$
$$ f_1\otimes \cdots \otimes f_n \longrightarrow  \lambda_n (f_1\otimes \cdots \otimes f_n)$$
such that for $v_1 \otimes \cdots \otimes  v_n\in
V^{\otimes n}$
$$\lambda_n (f_1\otimes \cdots \otimes f_n)(v_1 \otimes \cdots
\otimes v_n)=f_1(v_1) \otimes \cdots \otimes
f_n(v_n)$$
and set $$ \mu :=\Delta^{\star}\circ \lambda_2
\quad \quad \alpha :=
 \beta ^{ \ast}$$
where the star $\star$ denotes the transpose
linear map. Then, the quadruple $\left(
V^{\star},\mu ,\eta , \alpha\right) $ is a
 $G$-Hom-associative algebra. Indeed,  $\mu (f_1,f_2)=\mu_\K \circ
 \lambda_2 (f_1\otimes f_2) \circ \Delta $ where $\mu_\K$ is the multiplication of $\K$ and $f_1,f_2 \in
 V^{\ast}$. One has
\begin{eqnarray}\mu \circ ( \mu \otimes \alpha) (f_1 \otimes f_2
\otimes f_3 )&=& \mu
 (\mu (f_1 \otimes f_2) \otimes \alpha (f_3)) \nonumber \\ &=&\mu_\K \circ \lambda_2(\mu (f_1 \otimes f_2) \otimes \alpha
 (f_3))\circ \Delta \nonumber \\ &=& \mu_\K \circ \lambda_2 (\lambda_2( (f_1\otimes f_2) \circ
 \Delta)\otimes \alpha
 (f_3))\circ \Delta \nonumber  \\ &=&
 \mu_\K \circ (\mu_\K \otimes id ) \circ \lambda_3 (f_1 \otimes f_2
 \otimes f_3) \circ (\Delta \otimes \beta )\circ
 \Delta . \nonumber
\end{eqnarray}
Similarly
$$\mu \circ ( \alpha \otimes \mu ) (f_1 \otimes f_2
\otimes f_3 )= \mu_\K \circ ( id\otimes  \mu_\K)
\circ \lambda_3 (f_1 \otimes f_2 \otimes f_3)
\circ ( \beta\otimes\Delta  )\circ \Delta.$$
 Using the associativity and the commutativity of $\mu_\K$, the
$\alpha$-associator may be written as
 $$a_{\alpha , \mu}=\mu_\K \circ ( id\otimes  \mu_\K) \circ \lambda_3 (f_1
\otimes f_2
 \otimes f_3)\circ(  ( \Delta\otimes\beta  )\circ
 \Delta- ( \beta\otimes\Delta  )\circ
 \Delta).
 $$

 Then we have the following connection between the
 $\alpha$-associator and $\beta$-coassociator
$$a_{\alpha , \mu}=\mu_\K \circ ( id\otimes  \mu_\K) \circ \lambda_3 (f_1
\otimes f_2
 \otimes f_3)\circ \textbf{c}_\beta (\Delta).
 $$
 Therefore if $\left( V,\Delta ,\beta \right) $ is a $G$-Hom-coalgebra, then the   $\left(
V^\star,\Delta^\star ,\beta^\star \right) $ is a
$G$-Hom-associative algebra.

\end{proof}

\begin{proposition}
Let  $\left( V,\mu ,\alpha \right) $ be a finite
dimensional $G$-Hom-associative algebra where $G$
is a subgroup of $\mathcal{S}_3$. Its dual vector
space $V^\star$ is provided with a
$G$-Hom-coalgebra  $\left( V^\star,\mu^\star
,\alpha^\star \right) $, where $\mu^\star
,\alpha^\star$ are the transpose map.
 \end{proposition}
\begin{proof}
Let $\A=\left( V,\mu,\alpha  \right) $ be a
$n$-dimensional Hom-associative algebra ($n$
finite). Let $\{e_1,\cdots,e_n\}$ be a basis of
$V$ and $\{e^*_1,\cdots,e^*_n\}$ be the dual
basis. Then $\{e^*_i \otimes e^*_j\}_{i,j}$ is a
basis of $\A^\star \otimes \A^\star$. The
comultiplication $\Delta=\mu^\star$ on  $\A^\ast$
is defined for $f\in \A^\star$ by
$$\Delta (f)=\sum_{i,j=1}^{n}{f(\mu(e_i \otimes e_j))\ e^*_i \otimes e^*_j }
$$
Setting $\mu(e_i \otimes e_j)=\sum_{k=1}^n
{C_{ij}^k e_k}$ and $\alpha(e_i)=\sum_{k=1}^n
{\alpha_{i}^k e_k}$, then $\Delta
(e^*_k)=\sum_{i,j=1}^n{C_{ij}^k\ e^*_i \otimes
e^*_j } $ and $\beta(e_i)=\alpha^\star(e_i)
=\sum_{k=1}^n {\alpha_{k}^i e_k}$.

The condition \eqref{admi} of
$G$-Hom-coassociativity of  $\Delta$, applied to
any element $e^*_k$ of the basis, is equivalent
to
$$\sum_{p,q,s=1}^{n}{\sum_{\sigma\in G}{(-1)^{\epsilon (\sigma)} (\sum_{i,j=1}^{n}{\alpha_s^j C_{ij}^k C_{pq}^i -
\alpha _p^i C_{ij}^k C_{qs}^j})}e^*_{\sigma
^{-1}(p)} \otimes e^*_{\sigma ^{-1}(q)}\otimes
e^*_{\sigma {-1}(s)}}=0
$$
Therefore $\Delta$ is $G$-Hom-coassociative if
for  any $p,q,s,k\in\{1,\cdots, n\}$ one has
$$\sum_{\sigma\in G}{(-1)^{\epsilon (\sigma)} (\sum_{i,j=1}^{n}{\alpha_s^j C_{ij}^k C_{pq}^i -
\alpha _p^i C_{ij}^k C_{qs}^j})}=0$$ The previous
system is exactly the condition \eqref{G_ass} of
$G$-Hom-associativity of $\mu$, written on
$e_{p'}\otimes e_{q'}\otimes e_{s'}$ and setting
$p=\sigma (p'),\ q=\sigma (q'),\ s=\sigma (s')$.
\end{proof}
\begin{corollary}
The dual vector space of a Hom-coassociative
coalgebra $\left( V,\Delta ,\beta,\varepsilon
\right) $ is a  Hom-associative algebra $\left(
V^\star,\Delta^\star
,\beta^\star,\varepsilon^\star \right) $, where
$V^\star$ is the dual vector space and the star
for the linear maps denotes the transpose map.
The dual vector space of finite-dimensional
Hom-associative algebra is a Hom-coassociative
coalgebra.
\end{corollary}
\begin{proof}It is a particular case of the previous Propositions ($G=G_1$).
\end{proof}

\begin{remark} Let $V$ be a finite-dimensional $\K$-vector space.
If $H=(V,\mu ,\alpha,\eta ,\Delta ,
\beta,\varepsilon ,S)$ is a Hom-Hopf algebra,
then
$$H^*=(V^*,\Delta^* ,\beta^*,\varepsilon^*, \mu^*
, \alpha^*,\eta^* , S^*)$$ is also a Hom-Hopf
algebra.
\end{remark}

\begin{remark}
An earlier version of this article has appeared as a Preprint in Mathematical Sciences in Lund University, Center for Mathematical Sciences in 2008. The results in this paper were also presented in the talks given by the authors at the European Science Foundation Mathematics Conference "Algebraic Aspects in Geometry",
Mathematical Research and Conference Center,
Bedlewo, Poland, October 2007; AGMF Baltic-Nordic network conferences, G{\"o}teborg, Sweden, October 2007 and
Tartu, Estonia, October 2008;
International Conference on Noncommutative Rings and Geometry, Almeria, Spain, September 2007 and Seminar Sophus Lie XXXV conference, Budapest, Hungary, March, 2008.
\end{remark}


\newpage
\begin{thebibliography}{99}

\bibitem{AizawaSaito}
Aizawa, N., Sato, H., $q$-deformation of the Virasoro algebra with central extension,
Physics Letters B, Phys. Lett. B
\textbf{256}, no. 1, 185--190 (1991).
Aizawa, N., Sato, H., Hiroshima University preprint
preprint HUPD-9012 (1990).
\bibitem{ChaiElinPop} Chaichian, M. , Ellinas D. and
Z. Popowicz: Quantum conformal algebra with
central extension, Phys. Lett. B \textbf{248}, no. 1-2, 95--99 (1990).
\bibitem{ChaiKuLukPopPresn} Chaichian, M.,
Isaev, A. P. ,  Lukierski, J., Popowic, Z. and
Pre\v{s}najder P.: $q$-deformations of Virasoro
algebra and conformal dimensions, Phys. Lett. B
\textbf{262}   (1), 32--38 (1991).
\bibitem{ChaiIsKuLuk}  Chaichian, M., Kulish, P.,
Lukierski, J. : $q$-deformed Jacobi identity,
$q$-oscillators and $q$-deformed
infinite-dimensional algebras, Phys. Lett. B
\textbf{237} , no. 3-4, 401--406 (1990).
\bibitem{ChaiPopPres} Chaichian, M.,
Popowicz, Z. , Pre\v{s}najder, P. : $q$-Virasoro
algebra and its relation to the $q$-deformed KdV
system, Phys. Lett. B \textbf{249}, no. 1,
63--65 (1990).
\bibitem{CurtrZachos1} Curtright, T. L.,   Zachos, C. K.:
Deforming maps for quantum algebras, Phys. Lett.
B \textbf{243}, no. 3, 237--244  (1990).
\bibitem{DamKu} Damaskinsky, E. V.,  Kulish, P. P. :
Deformed oscillators and their applications,
(Russian),  Zap.\ Nauch.\ Semin.\ LOMI
\textbf{189}, 37--74 (1991); Engl. transl. in J.
Soviet Math.  \textbf{62}  (5),
2963--2986 (1992).
\bibitem{DaskaloyannisGendefVir}
Daskaloyannis, C. : Generalized deformed Virasoro
algebras, Modern Phys. Lett. A \textbf{7} no. 9, 809--816 (1992).
\bibitem{DelEtinQuantFieldStrings}
Deligne, P. , Etingof, P. ,Freed,  D. S. ,
Jeffrey, L. C., Kazhdan, D. , Morgan, J. W. ,
Morrison, D. R., Witten, E.  (Eds): Quantum fields and
strings: a course for mathematicians, 2 vol.
Amer. Math. Soc., (1999).
\bibitem{Drinfeld}  Drinfel'd V. G.: Hopf algebras and the quantum Yang-Baxter equation,  Soviet Math. Doklady, \textbf{32},   254--258 (1985).
\bibitem{DiFranMathiSen} Di Francesco, P. , Mathieu,
P. , S\'en\'echal, D. : Conformal field theory, 890 pp , Springer
Verlag, (1997).
\bibitem{FLM} Frenkel, I. , Lepowsky, J. ,  Meurman, A.:
Vertex operator algebras and the Monster, 508 pp
Academic Press, (1988).
\bibitem{Fuchs1} Fuchs, J. :
Affine Lie algebras and quantum groups, 433 pp, Cambridge
University Press, (1992).
\bibitem{Fuchs2} Fuchs, J. :
Lectures on conformal field theory and Kac-Moody
algebras, Springer Lecture Notes in Physics 498, 1--54
(1997).
\bibitem{GR}  Goze, M., Remm, E.:
Lie-admissible coalgebras , J. Gen. Lie
Theory Appl. {\bf 1}, no. 1, 19--28 (2007).
\bibitem{Guichardet}  Guichardet, A.: groupes quantiques, 149 pp, InterEditions / CNRS Editions (1995).
\bibitem{HLS} Hartwig, J. T., Larsson, D., Silvestrov, S. D.:
Deformations of Lie algebras using
$\sigma$-derivations, J. Algebra \textbf{295},  314--361
(2006).\ (original preprint appeared
in arXiv:math/0408064v1 [math.QA])
\bibitem{HellstromGrInv} Lars Hellstr{\"o}m,
A Rewriting Approach to Graph Invariants, 47-67, Chapter 5, in
{\rm S. Silvestrov, E. Paal, V. Abramov, A. Stolin,
(Eds.), Generalized Lie theory in Mathematics, Physics and Beyond, Springer-Verlag, Berlin, Heidelberg, (2008).}
\bibitem{HellstromGDL} Lars Hellstr{\"o}m,
A Generic Framework for Diamond Lemmas,
arXiv:0712.1142v1 [math.RA].
\bibitem{HelSil-book} Hellstr{\"o}m, L.,
Silvestrov, S. D. : Commuting elements in
$q$-deformed Heisenberg algebras,  256 pp.  World
Scientific, Singapore (ISBN:
981-02-4403-7) (2000).
\bibitem{Hu}  Hu, N.,: $q$-Witt algebras,
$q$-Lie algebras, $q$-holomorph structure and
representations,  Algebra Colloq. {\bf 6} ,
no. 1, 51--70 (1999).
\bibitem{Jakob}  Jakobsen, H. P.:
Matrix chain models and their $q$-deformations,
in ``Lie theory and its applications in physics
V'', World Sci. Publ., River Edge, NJ, 377--391 (2004). Appeared first in different version as
Preprint Mittag-Leffler Institute, Report No 23,
2003/2004, ISSN 1103-467X, ISRN IML-R-
-23-03/04-SE.
\bibitem{JakobLee}  Jakobsen, H. P. , Lee, H. C.-W.:
Matrix chain models and Kac-Moody algebras, in
``Kac-Moody Lie algebras and related topics",
  Contemp. Math. 343, Amer. Math. Soc., Providence, RI, 147--165
   (2004).
\bibitem{Kassel1} Kassel, C., \emph{Cyclic homology of differential operators,
the Virasoro algebra and a $q$-analogue}, Commun. Math. Phys. 146 (1992), 343-351.
\bibitem{Kabook}  Kassel, C.:
Quantum groups , 531 pp Graduate Texts in
Mathematics, 155. Springer-Verlag, New York
(1995).
\bibitem{LS1} Larsson, D., Silvestrov, S. D.:
Quasi-hom-Lie algebras, Central Extensions
and 2-cocycle-like identities, J. Algebra
\textbf{288},  321--344 (2005).
\bibitem{LS2} Larsson, D., Silvestrov, S. D.:
Quasi-Lie algebras, in "Noncommutative
Geometry and Representation Theory in
Mathematical Physics", Contemp. Math.,
\textbf{391}, Amer. Math. Soc., Providence, RI,  241--248
(2005).
\bibitem{LS3} Larsson, D., Silvestrov, S. D.:
Quasi-deformations of $sl_2(\mathbb{F})$
using twisted derivations, Comm. in Algebra
\textbf{35},  4303 -- 4318 (2007). \ (original
preprint appeared in arXiv:math/0506172v2
[math.RA]).
\bibitem{LarsSilvSig-1-sl2}
Larsson, D., Sigurdsson, G., Silvestrov, S. D.,
On some almost quadratic algebras coming from
twisted derivations, J. Nonlinear Math. Phys.
\textbf{13} (2006), suppl. 1, 76--86.
\bibitem{LSSAGMFGoet}
Larsson, D., Sigurdsson, G., Silvestrov, S. D.:
Quasi-Lie deformations on the algebra
$\F[t]/(t^N)$, to appear in Proceedings of AGMF
Baltic-Nordic network conference, G{\"o}teborg,
October, (2007).
\bibitem{LeBruyn} Le Bruyn, L.:
Conformal $sl_2$ enveloping algebras, Comm. Alg.
\textbf{23} no. 4, 1325--1362 (1995).
\bibitem{LeBruynSmith}  Le Bruyn, L.,Smith, S. P.:
Homogenized $sl_2$, Proc. AMS 118,
725--730 (1993).
\bibitem{LeBruynSmithvdBergh}
Le Bruyn, L. , Smith, S. P. , Van den Bergh, M.:
Central extensions of three-dimensional
Artin--Schelter regular algebras, Math. Zeit.
\textbf{222} no. 2, 171--212 (1996).
\bibitem{LeBruynvdBergh}
Le Bruyn, L. , Van den Bergh, M. : On quantum spaces
of Lie algebras, Proceedings American
Mathematical Society, \textbf{119},  407--414 (1993) .
\bibitem{LiuKeQin} Liu, Ke Qin:
Characterizations of the quantum Witt algebra,
Lett. Math. Phys.  \textbf{24} , no. 4, 257--265 (1992).
\bibitem{Loday} Loday, J.L.: Generalized bialgebras and triples of operads", 110 pp, ArXiv:math.QA/0611885.
\bibitem{Makhlouf-Hopf}  Makhlouf,  A.:
 Degeneration, rigidity and irreducible
components of Hopf algebras,  Algebra
Colloquium, vol  \textbf{12} (2),  241--254 (2005).
\bibitem{Makhlouf-Hermann}  Makhlouf,  A.:
 Alg\`{e}bre de Hopf et renormalisation en th\'{e}orie quantique des champs,  In "th\'{e}orie quantique des champs : M\'{e}thodes et Applications", Travaux en  Cours, Hermann Paris,    191--242 (2007).
\bibitem{MS} Makhlouf, A., Silvestrov, S. D.:
 Hom-algebra structures, J. Gen. Lie Theory Appl. \textbf{2} (2) , 51--64 (2008).
\bibitem{HomHopf} Makhlouf, A., Silvestrov, S. D.:
\emph{Hom-Lie admissible Hom-coalgebras and Hom-Hopf algebras}, Preprints
in Mathematical Sciences, Lund University, Centre
for Mathematical Sciences, Centrum Scientiarum
Mathematicarum, (2007), arXiv:0709.2413 [math.RA] (2007).
Published also as Chapter 17, pp 189-206,
{\rm S. Silvestrov, E. Paal, V. Abramov, A. Stolin,
(Eds.), Generalized Lie theory in Mathematics, Physics and Beyond, Springer-Verlag, Berlin, Heidelberg, (2008).}
\bibitem{HomDeform} Makhlouf, A., Silvestrov, S. D.:
\emph{Notes on Formal deformations of
Hom-Associative and Hom-Lie algebras}, To appear in Forum Mathematicum. Preprints
in Mathematical Sciences, Lund University, Centre
for Mathematical Sciences, Centrum Scientiarum
Mathematicarum, (2007:31) LUTFMA-5095-2007, 2007;
arXiv:0712.3130v1 [math.RA] (2007).
\bibitem{MontgomeryLivre}  Montgomery, S.:  Hopf algebras and their actions on rings, 238 pp,  CBMS Regional Conference Series in Mathematics,
\textbf{82}, AMS Providence, Rhode Island (1993).
\bibitem{RichardSil} Richard L., Silvestrov S. D.:
Quasi-Lie structure of $\sigma$-derivations
of $\mathbb{C}[t^{\pm 1}]$, J. Algebra, {\bf 319}, no 3, 1285-1304, (2008).
\bibitem{RS1AGMFLund} Richard, L., Silvestrov, S. D., 
{\rm A Note On Quasi-Lie and Hom-Lie Structures
of $\sigma$-Derivations of ${\mathbb C}[z_1^{\pm
1},\ldots,z_n^{\pm 1}]$}, Chapter 22, pp 257--262, 
{\rm S. Silvestrov, E. Paal, V. Abramov, A. Stolin,
(Eds.), Generalized Lie theory in Mathematics, Physics and Beyond,
Springer-Verlag, Berlin, Heidelberg, 2008.}
\bibitem{FVOlaggeomassocalgbook} Van Oystaeyen, F. :
Algebraic geometry for associative algebras, vi+287 pp,
Monographs and Textbooks in Pure and Applied
Mathematics, 232. Marcel Dekker, Inc., New York, ISBN: 0-8247-0424-X.
(2000).
\bibitem{Yau:EnvLieAlg} Yau D.:
Enveloping algebra of Hom-Lie algebras, J. Gen.
Lie Theory Appl. \textbf{2} (2), 95--108 (2008).
\bibitem{Yau:HomolHomLie} Yau D.:
 Hom-algebras as deformations and homology,
arXiv:0712.3515v1 [math.RA] (2007).
\bibitem{Yau:HomBial} Yau D.:
 Hom-bialgebras and comodule algebras,
arXiv:0810.4866v1 [math.RA] (2008).
\end{thebibliography}
\end{document}